\documentclass[a4paper,12pt,reqno]{amsart}
\usepackage{amsmath,amsfonts,amssymb,amsthm,enumerate}
\usepackage{tikz,graphicx}
\usepackage{hyperref}
\usepackage{cleveref}
\usepackage{caption}
\usepackage{enumitem}
\usepackage{float,wrapfig}
\usepackage[labelsep=space]{caption}
\usepackage[font=footnotesize]{caption}

\newtheorem{theorem}{Theorem}[section]
\newtheorem{proposition}{Proposition}[section]
\newtheorem{definition}{Definition}[section]
\newtheorem{example}{Example}[section]
\newtheorem{corollary}{Corollary}[section]
\newtheorem{remark}{Remark}[section]

\numberwithin{equation}{section}
\DeclareCaptionLabelFormat{andtable}{#1~#2  \&  \tablename~\thetable}
\title[]{SPECTRAL TRANSFORMATION ASSOCIATED WITH A PERTURBED $R_I$ TYPE RECURRENCE RELATION }
\author{Vinay Shukla$^\dagger$}
\address{$^\dagger$Department of Mathematics\\ Indian Institute of Technology, Roorkee-247667, Uttarakhand, India}
\email{vshukla@ma.iitr.ac.in}
\author{A. Swaminathan\, $^{\#\ddagger}$}{\thanks{$^{\#}$Corresponding author}}
   \address{$^\ddagger$Department of Mathematics\\ Indian Institute of Technology, Roorkee-247667, Uttarakhand, India}
   \email{mathswami@gmail.com, a.swaminathan@ma.iitr.ac.in}
\allowdisplaybreaks
\bigskip
\allowdisplaybreaks

\begin{document}
\leftline{ \scriptsize \it  }
\keywords{Orthogonal Polynomials; Spectral transformation, $R_I$ type recurrence Relation; Continued fractions; Interlacing of zeros; Carath\'eodary function.}
\subjclass[2020] {42C05, 30B70, 30C15}
\maketitle

\markboth{Vinay Shukla and A. Swaminathan}{Perturbed $R_{I}$ type recurrence relation}

\begin{abstract}
In this work, orthogonal polynomials satisfying $R_I$ type recurrence relation %$\mathcal{P}_{n+1}(z) = (z-c_n)\mathcal{P}_n(z)-\lambda_n (z-a_n)\mathcal{P}_{n-1}(z),$ with $\mathcal{P}_{-1}(z) = 0$ and $\mathcal{P}_0(z) = 1$ 
are analyzed when the recurrence coefficients are modified. The structural relationship between the perturbed and the unperturbed polynomials along with the spectral properties and spectral transformation of continued fraction are investigated. It is demonstrated that the transfer matrix method is computationally more efficient than the classical method for obtaining perturbed $R_I$ polynomials. Further, an interesting consequence of co-dilation on the Carath\'eodary function is presented. Finally, the study of co-recursion and co-dilation in connection to the unit circle is carried out with the help of an illustration. The interlacing and monotonicity of zeros between L-Jacobi polynomials and their perturbed forms are demonstrated. 
\end{abstract}

\section{Introduction}
Consider the recurrence relation
\begin{align}\label{R1}
&  \mathcal{P}_{n+1}(z) = (z-c_n)\mathcal{P}_n(z)-\lambda_n (z-a_n)\mathcal{P}_{n-1}(z), \quad n\geq 0, \\ %\quad \lambda_n \neq 0 ,\quad n\geq 1  
& \nonumber	\mathcal{P}_{-1}(z) = 0, \qquad \mathcal{P}_0(z) = 1. %\qquad \mathcal{P}_n(a_n) \neq 0
\end{align}
It was demonstrated in \cite[Theorem 2.1]{Esmail masson JAT 1995} that if $\lambda_n \neq 0$ and $\mathcal{P}_n(a_n)\neq 0$ for $n\geq 1$, then there exists a rational function $\psi_n(z)=\frac{\mathcal{P}_n(z)}{\prod_{j=1}^{n}(z-a_j)}$ and a linear functional $\mathfrak{N}$ such that the orthogonality relations
\begin{align*}
\mathfrak{N}[a_0]\neq 0, \quad \mathfrak{N}\left[z^k \psi_n(z) \right]\neq  0, \quad 0\leq k < n,
\end{align*}
hold. We can obtain \eqref{R1} by starting with a sequence of rational functions $\psi_n(z)$ that satisfy a three-term recurrence relation and have poles at $\{a_k\}_{k=0}^\infty$. According to \cite{Esmail masson JAT 1995}, the recurrence relation \eqref{R1} is a $R_I$ type recurrence relation, and $\mathcal{P}_n(z), n\geq 1$, generated by \eqref{R1} are $R_I$ polynomials. The $R_I$ type recurrence relation \eqref{R1} has several applications such as defining biorthogonal systems of rational functions, studying multipoint rational interpolants or Pade approximation, evaluating some types of beta integrals, and so on. \par %and associated continued fraction as $R_I$-fractions. 
%The application of the Gram-Schmidt orthonormalization process in the linear space of rational functions yields a sequence of orthonormal rational functions.  \par
Note that the infinite continued fraction
\begin{align}\label{R1 continued fraction}
R_I(z) = \frac{1}{z-c_0} \mathbin{\genfrac{}{}{0pt}{}{}{-}} \frac{\lambda_1(z-a_1)}{z-c_1} \mathbin{\genfrac{}{}{0pt}{}{}{-}} \frac{\lambda_2(z-a_2)}{z-c_2} \mathbin{\genfrac{}{}{0pt}{}{}{-}} \mathbin{\genfrac{}{}{0pt}{}{}{\cdots}} ,
\end{align}
becomes finite when $z=a_k$, $k \geq 2$. Following \cite{Dini Maroni Ronveaux 1989}, we call it a $R_I$-fraction. Polynomials $\mathcal{P}_n(z)$ are the numerator polynomials associated with \eqref{R1 continued fraction}. Also, the polynomials of the second kind \cite[page 7, equation (2.9)]{Esmail masson JAT 1995} associated with recurrence relation \eqref{R1} are given by
\begin{align}\label{Associted polynomials Q_n}
&  \mathcal{Q}_{n+1}(z) = (z-c_n)\mathcal{Q}_n(z)-\lambda_n (z-a_n)\mathcal{Q}_{n-1}(z), \quad n\geq 1, \\ %\quad \lambda_n \neq 0 ,\quad n\geq 1  
& \nonumber	\mathcal{Q}_{0}(z) = 0, \qquad \mathcal{Q}_1(z) = 1,
\end{align}
such that the ratio $\dfrac{\mathcal{Q}_n(z)}{\mathcal{P}_n(z)}$ is the $n$-th convergent of the continued fraction \eqref{R1 continued fraction}. The existence of an $m-$function associated with an $R_I$-fraction was established in \cite[Theorem 2.3]{Esmail masson JAT 1995}. The orthogonal Laurent polynomials (OLP) or L-orthogonal polynomials \cite{dimitrov ranga 2002} are the $R_I$ polynomials \eqref{R1} when $a_n=0$ and $z$ is real. They satisfy the recurrence
\begin{align}\label{special R1}
&\mathcal{P}_{n+1}(x) = (x-c_n)\mathcal{P}_n(x)-\lambda_n x\mathcal{P}_{n-1}(x), \quad n\geq 0, \quad x \in \mathbb{R}, \\    
& \nonumber	\mathcal{P}_{-1}(x) = 0, \qquad \mathcal{P}_0(x) = 1,
\end{align}
where $\{c_n\}_{n \geq 0}$ and $\{\lambda_n\}_{n \geq 0}$ are positive numbers. For more details, we refer to \cite{Simon part1 2005} and the references therein. \cite{Kim stanton 2021} provides a combinatorial interpretation of $R_I$ polynomials. Interested readers may look at \cite{Bracciali ranga swami paraorthogonal 2016,Castillo costa ranga veronese Favard theorem 2014,Costa felix ranga 2013} for some recent progress related to $R_I$ type recurrence relations and $R_I$ polynomials. \par 

% \section{Motivation of the problem}
The study of perturbation of the coefficients of a recurrence relation has considerable literature and has been dealt with in numerous ways by many authors. Constructing new sequences by modifying the original sequence is a powerful tool with many applications to theoretical and physical problems. \par 
Orthogonal polynomials on the real line (OPRL) satisfy the recurrence relation
\begin{align}\label{OPRL TTRR R1 paper}
\mathcal{R}_{n+1}(x) = (x-b_{n})\mathcal{R}_n(x)-\gamma_n \mathcal{R}_{n-1}(x), \quad \mathcal{R}_{-1}(x) = 0, \quad \mathcal{R}_0(x) = 1, \quad n \geq 0.
\end{align}
The case, known as co-recursive, was introduced and studied in \cite{chihara PAMS 1957} by adding a constant to the first coefficient $b_0$ . The case of another type of perturbation called co-dilation was presented in \cite{Dini thesis 1988}. The Stieltjes function and the fourth-order differential equation for the co-dilated polynomials inside the Laguerre-Hahn class were investigated in \cite{Dini Maroni Ronveaux 1989, Ronveaux vigo 1988}. The general case, called the generalized co-modification, arising from perturbation of coefficients in \eqref{OPRL TTRR R1 paper} at any level, was studied in \cite{Paco perturbed recurrence 1990} from the point of view of Stieltjes function, fourth-order differential equation, and distribution of zeros. A connection to the Laguerre-Hahn class was also investigated in \cite{Paco perturbed recurrence 1990}. The properties of co-modified classical orthogonal polynomials were studied in \cite{Ronveaux vigo 1988}. Its extension to semi-classical orthogonal polynomials was investigated in \cite{Dini Maroni Ronveaux 1989, ronveaux belmehdi dini 1990}. Interlacing properties and some new inequalities for the zeros of co-modified OPRL, called co-polynomials on the real line (COPRL), were studied in \cite{Castillo co-polynomials on real line 2015}.
%Such perturbations are not artificial because they are linked to perturbation in the entries of the corresponding Jacobi matrix, which in turn, is useful in further applications like the study of Hamiltonian operator appearing in quantum mechanics in the study of many body systems. Motivated by this problem, the co-recursive, co-dilated and co-modified polynomials have been introduced and their properties have been studied. Later, while some experiments with lasers, it was needed to make perturbations anywhere on the surface of the target many-body system and to study how the spectroscopic properties gets changed. This led to the study of the generalized co-recursive, co-dilated and co-modified polynomials \cite{Paco perturbed recurrence 1990}. 
In this direction, finite perturbations were studied in \cite{Peherstorfer 1992}. When the perturbations are large, some information is expected to be lost, which was explored in \cite{Leopold 1 2003}. For perturbations of recurrence coefficients in the recurrence relations of higher-order and their extensions to Sobolev OPRL, see \cite{Leopold 2 2007, Leopold 3 2008}. Recently, a transfer matrix approach was introduced in \cite{Castillo co-polynomials on real line 2015} to study polynomials perturbed in a generalized co-modified way. \par 
Unlike OPRL, a sequence of monic orthogonal polynomials on the unit circle (OPUC) denoted by $\{\phi_n(z)\}_{n=0}^\infty$ satisfies a first-order recurrence relation (Szeg\H{o} recurrence)
\begin{align}\label{Szego recurrence}
\begin{bmatrix}
	\phi_{n+1}(z)	\\
	\phi^*_{n+1}(z)
\end{bmatrix} & = T_n(z) \begin{bmatrix}
	\phi_{n}(z)	\\
	\phi^*_{n}(z)
\end{bmatrix}, \quad T_n(z) = \begin{bmatrix}
	z &	-\overline{\alpha_n} \\ 
	-\alpha_n z & 1
\end{bmatrix},
\end{align}
with initial condition $\phi_0(z)=1$, where $\phi^*_{n}(z)=z^n \overline{\phi_n(\frac{1}{\bar{z}})}$ is the reversed polynomial, and $T_n(z)$ is the transfer matrix \cite{Simon part1 2005}. The elements of the sequence $\{\alpha_n\}_{n=0}^\infty$ where $\alpha_n = -\overline{\phi_{n+1}(0)}$ lie in the unit disc and are known as Schur, reflection, Geronimus or Verblunsky coefficients. \par 
The monic OPUC is completely determined by its reflection coefficients (Verblunsky theorem). This fact motivated the authors in \cite{Castillo perturbed szego 2014} to study polynomials associated with perturbations of Verblunsky coefficients. Again, a transfer matrix approach has been used to study so-called co-polynomials on the unit circle (COPUC). The structural relations and rational spectral transformation for C-functions associated with COPUC are also discussed in \cite{Castillo perturbed szego 2014} (see also \cite{Castillo perturbed polynomials via szego transform 2017}). \par  
In this direction, results related to co-recursive of $d-$ orthogonal polynomials \cite{saib zerouki 2013},  co-recursive of $q-$ classical orthogonal polynomials \cite{Foupouagnigni Ronveaux 2001}, fourth-order difference equation and its factorization for co-dilated and/or co-recursive of classical discrete orthogonal polynomials \cite{Foupouagnigni Koepf Ronveaux 2003} are available in the literature. For some applications, see \cite{Erb 2012, Erb 2015, Slim 1988}. \par 
This manuscript aims to study properties of the polynomials that satisfy a recurrence relation as \eqref{R1} with new recurrence coefficients perturbed in a (generalized) co-recursive/co-dilated/co-modified way, i.e.,
\begin{align}\label{R1 perturbed}
	\mathcal{P}_{n+1}(z;\mu_k,\nu_k) = (z-c^*_n)\mathcal{P}_n(z;\mu_k,\nu_k)-\tilde{\lambda}_n (z-a_n)\mathcal{P}_{n-1}(z;\mu_k,\nu_k), \quad n\geq 1,
\end{align}
with initial conditions $\mathcal{P}_0(z) = 1$ and $\mathcal{P}_{-1}(z) =0$. In other words, we consider arbitrary single modification of recurrence coefficients as follows:
\begin{align}
	c^*_n &= c_n+\mu_k \delta_{n,k}, \quad \mu_k \in \mathbb{R}, \qquad \text{(co-recursive)} \label{co-recursive condition R1}\\
	\tilde{\lambda}_n &= \nu_k^{\delta_{n,k}} \lambda_n, \quad \nu_k>0, \qquad \text{(co-dilated)} \label{co-dilated condition R1}
\end{align}
where $k$ is a fixed non-negative integer. The rest of the paper is summarized as follows: \Cref{The continued fraction} defines co-polynomials of $R_I$ type and analyzes rational spectral transformations associated with the corresponding continued fraction. Structural relations between the perturbed polynomials and the original ones are developed via two methods: the transfer matrix method and the classical matrix method. The computational complexities of both methods for the generation of perturbed $R_I$ polynomials are compared. The obtained results are compared with the obtained results for perturbed $R_{II}$ polynomials in \cite{swami vinay R2 2022}, and some interesting conclusions are drawn. Further, the behaviour of zeros of L-orthogonal polynomials and their perturbed versions are studied. Finally, a result for obtaining the parameters of the PPC-fraction corresponding to a perturbed chain sequence is presented. In \Cref{Illustrative examples}, an illustration is provided for exploring the effects of co-recursion and co-dilation in the chain sequences on Szeg\H{o} polynomials, reflection coefficients, and associated measures. L-Jacobi polynomials are used to illustrate interlacing and monotonicity results on zeros numerically. % Section \ref{A $m-$ function point of view} discusses a rational spectral transformation. occuring in the special form of $R_I$ recurrence, 

\section{co-polynomials of $R_I$ type and Spectral Transformation}\label{The continued fraction}
\subsection{Structural relations}\label{Structural relations} Perturbing the coefficients $c_n$ and $\lambda_n$ at $n=k$ according to Favard theorem \cite[Theorem 2.1]{Esmail masson JAT 1995} generates a new family of $R_I$ polynomials. A more general situation in the finite case can be any change at any level. The recurrence relation of $R_I$ type for generalised co-modified polynomials are given by
\begin{align}\label{Generalised comodified equations}
	 \mathcal{P}_{k+1}(z;\mu_k,\nu_k) &= (z-c_k-\mu_k)\mathcal{P}_k(z;\mu_k,\nu_k)-\nu_k \lambda_k (z-a_k)\mathcal{P}_{k-1}(z;\mu_k,\nu_k),\quad \mbox{$n = k$} \nonumber \\,
	\mathcal{P}_{n+1}(z;\mu_k,\nu_k) &= (z-c_n)\mathcal{P}_n(z;\mu_k,\nu_k)-\lambda_n (z-a_n)\mathcal{P}_{n-1}(z;\mu_k,\nu_k), \quad \mbox{$n \neq k$}.
%	\mathcal{P}_{n+1}(z;\mu_k,\nu_k) &=(z-c_n)\mathcal{P}_n(z;\mu_k,\nu_k)-\lambda_n (z-a_n)\mathcal{P}_{n-1}(z;\mu_k,\nu_k),\quad \mbox{$n \geq k+1$}.
\end{align}
%A single modification of $c_n$ at the level $k$, called generalised co-recursive, modification of $\lambda_n$ at the level $k$, called generalised co-dilated and modification of both $c_n$ and $\lambda_n$ at the same level $k$, called generalised co-modified, are considered.
%The last recurrence relation can be solved in terms of $\mathcal{P}_{n+1}(z)$ and associated polynomials of order $r$, i.e., $\mathcal{P}^{(r)}_{n-r}(z)$ to obtain the representation of new perturbed polynomials in terms of unperturbed ones (see \Cref{Theorem s_k_x R1 pp}).
%\begin{align*}
%\mathcal{P}_{n+1}(z;\mu_k,\nu_k) &= \mathcal{P}_{n+1}(z)+[(1-\nu_k)\lambda_k (z-a_k) \mathcal{P}_{k-1}(z)-\mu_k \mathcal{P}_k(z)] \mathcal{P}^{(k)}_{n-k}(z), \qquad \mbox{$n \geq k$}, \\
%\mathcal{P}_{n+1}(z;\mu_k,\nu_k) &= \mathcal{P}_{n+1}(z), \qquad \mbox{$n < k$}.
%\end{align*}

The recurrence relation of $R_I$ type for generalised co-recursive polynomials $\mathcal{P}_{n+1}(z;\mu_k)$ and for generalised co-dilated polynomials $\mathcal{P}_{n+1}(z;\nu_k)$ and their representation in terms of unperturbed ones can be obtained by substituting $\nu_k=1$ and $\mu_k=0$, respectively, in \eqref{Generalised comodified equations}. 
\begin{remark} 
Co-polynomials on the real line (COPRL) \cite{Castillo co-polynomials on real line 2015} and co-polynomials on the unit circle (COPUC) \cite{Castillo perturbed szego 2014} are obtained after perturbations in \eqref{OPRL TTRR R1 paper} and in \eqref{Szego recurrence}. Following a similar terminology, we can refer to the perturbed polynomials introduced above as \textbf{co-polynomials of $R_I$ type}.
\end{remark}
Let us introduce
\begin{align*}
	\mathbb{F}_{n+1} & :=	\begin{bmatrix}
		\mathcal{P}_{n+1}(z)	& -\mathcal{Q}_{n+1}(z)\\
		\mathcal{P}_{n}(z)	& -\mathcal{Q}_{n}(z)
	\end{bmatrix} = \mathbf{T}_n \mathbb{F}_{n},
	%& = A_n \mathbb{B}_{n} \\
	%& = \begin{bmatrix}
	%	x-c_n & -\lambda_n (x-a_n)(x-b_n) \\
	%	1 & 0
	%\end{bmatrix} \begin{bmatrix}
	%	\mathcal{P}_{n}	& -\mathcal{Q}_{n}\\
	%	\mathcal{P}_{n-1}	& -\mathcal{Q}_{n-1}
	%\end{bmatrix} \nonumber
\end{align*}
where$\{\mathcal{Q}_n(z)\}_{n \geq 0}$ are the associated polynomials of second kind that satisfy the recurrence relation \eqref{special R1} with initial conditions $\mathcal{Q}_0(z) = 0$ and $\mathcal{Q}_1(z) = 1$. They are monic polynomials of degree $n-1$. Clearly, $\mathbb{F}_{n+1} $ can be written as the product of the transfer matrices
\begin{align}\label{F_n+1 to T_0 R1 pp}
	&	\mathbb{F}_{n+1} =  \mathbf{T}_n \mathbb{F}_{n} = \mathbf{T}_n \ldots \mathbf{T}_{k+1}\mathbf{T}_k \mathbf{T}_{k-1} \ldots \mathbf{T}_0,
\end{align}
where $\mathbf{T}_n$ is the transfer matrix given by
\begin{align*}
%	\mathbb{P}_{n+1} &=
%	\begin{bmatrix}
%		\mathcal{P}_{n+1}(x) & \mathcal{P}_n(x)	
%	\end{bmatrix}^T, \quad
	\mathbf{T}_n= \begin{bmatrix}
		z-c_n & -\lambda_n (z-a_n) \\
		1 & 0
	\end{bmatrix} \quad {\rm and} \quad \det(\mathbf{T}_n)=\lambda_n (z-a_n).
\end{align*}
%Let us consider
% Now, from \eqref{special R1}, we have
%\begin{align*}
%	\mathbb{P}_{n+1}&= \mathbf{T}_n \mathbb{P}_n = \begin{bmatrix}
%		x-c_n & -\lambda_n x \\
%		1 & 0
%	\end{bmatrix}
%	\begin{bmatrix}
%		\mathcal{P}_n(x) \\
%		{\mathcal{P}}_{n-1}(x)
%	\end{bmatrix}, 
	%= \begin{bmatrix}
	%	\mathcal{P}_{n+1} \\ \mathcal{P}_n
	%\end{bmatrix}
%\end{align*}
%\begin{align}\label{P_n+1 to P_0}
%	\mathbb{P}_{n+1} &= (\mathbf{T}_n \ldots \mathbf{T}_0) \mathbb{P}_{0}, \qquad
%	\mathbb{P}_{0} =
%	\begin{bmatrix}
%		\mathcal{P}_{0}(x) & \mathcal{P}_{-1}(x)	
%	\end{bmatrix}^T.
%\end{align}
Using matrix notation \eqref{F_n+1 to T_0 R1 pp}, we have
\begin{align}\label{P_n+1 mu_k nu_k R1 pp}
	\mathbb{F}_{n+1}(\mu_k,\nu_k) &= (\mathbf{T}_n \ldots \mathbf{T}_{k+1})\mathbf{T}_k(\mu_k,\nu_k)(\mathbf{T}_{k-1} \ldots \mathbf{T}_0) \mathbb{F}_{0},
\end{align}
where
\begin{align*}
	\mathbf{T}_k (\mu_k,\nu_k) & = \begin{bmatrix}
		z-c_k-\mu_k & -\nu_k \lambda_k (z-a_k) \\
		1 & 0	
	\end{bmatrix}.
\end{align*}
Further, $\mathbf{T}_k (\mu_k,\nu_k)$ can be written as
\begin{align}\label{T_k + N_k}
	\mathbf{T}_k (\mu_k,\nu_k)=\mathbf{T}_k+\mathbf{N}_k, \quad \mathbf{N}_k = \begin{bmatrix}
		-\mu_k & -(\nu_k-1) \lambda_k (z-a_k) \\
		0 & 0	
	\end{bmatrix}.
\end{align}
The associated polynomials of order $k+1$, $\mathcal{P}^{(k+1)}_{n-k}(x)$, satisfy 
\begin{align*}
\mathcal{P}^{(k+1)}_{n+1}(z) = (z-c_{n+k+1})\mathcal{P}^{(k+1)}_n(x)-\lambda_{n+k+1}(z-a_{n+k+1}) \mathcal{P}^{(k+1)}_{n-1}(z), \quad n\geq 0,
\end{align*}
with initial conditions $\mathcal{P}^{(k+1)}_{-1}(z)=0$ and $\mathcal{P}^{(k+1)}_{0}(z)=1$. Hence, by Favard's theorem \cite{Esmail masson JAT 1995}, there exists a moment functional with respect to which $\mathcal{P}^{(k+1)}_{n}(z)$ is also a sequence of $R_I$ polynomials. According to the theory of linear difference equations, the two sequences $\{u_n\}_{n \geq 0}$ and $\{v_n\}_{n \geq 0}$ are said to be linearly independent if the Casorati determinant
\begin{align}\label{casarotti determinant}
	D(u_n,v_n)=\begin{array}{|cc|}
		u_n & v_n \\
		{u}_{n+1} & {v}_{n+1}
	\end{array}
\end{align}
is different from zero for every $n$ \cite{milne 1951}. Since $\mathcal{P}^{(k+1)}_{n-k}(z)$ is a solution of the recurrence relation \eqref{special R1} with initial conditions $\mathcal{P}^{(k+1)}_{-1}(z)=0$ and $\mathcal{P}^{(k+1)}_{0}(z)=1$, it is easy to verify that 
\begin{align*}
	\begin{bmatrix}
		\mathcal{P}_{n+1}(z)	& \mathcal{P}^{(k+1)}_{n-k}(z)\\
		\mathcal{P}_{n}(z)	& \mathcal{P}^{(k+1)}_{n-k-1}(z)
	\end{bmatrix} &=\mathbf{T}_n \begin{bmatrix}
		\mathcal{P}_{n}(z)	& \mathcal{P}^{(k+1)}_{n-k-1}(z)\\
		\mathcal{P}_{n-1}(z)	& \mathcal{P}^{(k+1)}_{n-k-2}(z)
	\end{bmatrix}.
\end{align*}
%Hence,
%\begin{align*}
%	D(\mathcal{P}_{n}(x),\mathcal{P}^{(k)}_{n-k}(x))=\lambda_n xD(\mathcal{P}_{n-1}(x),\mathcal{P}^{(k)}_{n-k-1}(x)).
%\end{align*}
Let $X$ denote the set of zeros of $\mathcal{P}_{k}(z)$. From the above equality, we get
\begin{align}\label{D_Pn}
	D(\mathcal{P}_{n+1}(z),\mathcal{P}^{(k+1)}_{n-k}(z))&= \mathcal{P}_{k}(z)\prod_{j=k+1}^{n}\lambda_j(z-a_j),
\end{align}
which means that $\mathcal{P}_{n+1}(z)$ and $\mathcal{P}^{(k+1)}_{n-k}(z)$ are linearly independent in $\mathbb{C}\backslash X$.
\subsection{The spectral transformation} The Stieltjes or the Cauchy (also called the Borel \cite{Simon part1 2005} ) transform of the orthogonality measure $d\alpha$ is given by 
\begin{align}\label{m_alpha}
	m_\alpha(z)=\int_\Gamma \dfrac{d\alpha(y)}{z-y}, \quad z \in \mathbb{C}\backslash \Gamma .
\end{align}
An expression analogous to \eqref{m_alpha} for $R_{I}$ polynomials is given in \cite[eqn. 2.11, page 7]{Esmail masson JAT 1995}
\begin{align}\label{R1 integral}
	R_{I}(z)= \int_\Gamma \dfrac{d\alpha(y)}{z-y}, \quad z \in \mathbb{C}\backslash \Gamma,
\end{align}
where $\Gamma = \left(\displaystyle\bigcup_{j=1}^{N}\Gamma_j\right) \cup \{z_k\}_{k\geq 0}$. Here, $\{\Gamma_j\}_{j=1}^{N}$ denote the finite number of branch cuts and $\{z_k\}_{k\geq 0}$ be the number of poles of the function at which corresponding continued fraction converges. \par 
Also \cite[eqn 2.10]{Esmail masson JAT 1995} suggests that
\begin{equation}\label{R1_z continued fraction}
	\mathcal{R}_{I}(z) = \frac{1}{z-c_0} \mathbin{\genfrac{}{}{0pt}{}{}{-}} \frac{\lambda_1(z-a_1)}{z-c_1} \mathbin{\genfrac{}{}{0pt}{}{}{-}} \frac{\lambda_2(z-a_2)}{z-c_2} \mathbin{\genfrac{}{}{0pt}{}{}{-}} \mathbin{\genfrac{}{}{0pt}{}{}{\cdots}} 
	%&= \frac{1}{z-c_0-\dfrac{\lambda_1(z-a_1)}{z-c_1-\dfrac{\lambda_2(z-a_2)}{z-c_2- \ldots}}}
\end{equation}
and
%	&= \left[(x-c_0) - \frac{\lambda_1(x-a_1)}{x-c_1 -\mathbin{\genfrac{}{}{0pt}{}{}{\cdots}}} \right]^{-1} \\
\begin{equation}\label{R1_k+1_z continued fraction}
	\mathcal{R}_{I}^{k+1}(z) = \frac{1}{z-c_{k+1}} \mathbin{\genfrac{}{}{0pt}{}{}{-}} \frac{\lambda_{k+2}(z-a_{k+2})}{z-c_{k+2}} \mathbin{\genfrac{}{}{0pt}{}{}{-}} \frac{\lambda_{k+3}(z-a_{k+3})}{z-c_{k+3}} \mathbin{\genfrac{}{}{0pt}{}{}{-}} \mathbin{\genfrac{}{}{0pt}{}{}{\cdots}} 
	%&=\frac{1}{z-c_{k+1}-\dfrac{\lambda_{k+2}(z-a_{k+2})}{z-c_{k+2}-\dfrac{\lambda_{k+3}(z-a_{k+3})}{z-c_{k+3}- \ldots}}}
\end{equation}
%	&= \left[(x-c_{k+1}) - \frac{\lambda_{k+2}(x-a_{k+2})}{x-c_{k+2} -\mathbin{\genfrac{}{}{0pt}{}{}{\cdots}}} \right]^{-1} \\
are the representations of the continued fraction for $\mathcal{R}_{I}(z)$ and its $(k+1)$-th approximant, respectively. Integral transforms are widely used in the literature for studying various special functions, for some recent references in this regard, see \cite{Albayrak Dernek Ucar 2021, Bansal Nisar Singh 2020}.
%(continued fraction for $(k+1)$-th associated polynomial sequence). 

From \cite[chapter 4, equation 4.4]{Chihara book 1978}, we have
\begin{align}\label{chihara identity}
	\frac{A_{n+1}}{B_{n+1}} &= b_0+\frac{a_1}{b_1} \mathbin{\genfrac{}{}{0pt}{}{}{+}} \frac{a_2}{b_2} \mathbin{\genfrac{}{}{0pt}{}{}{+}} \mathbin{\genfrac{}{}{0pt}{}{}{\cdots}} \mathbin{\genfrac{}{}{0pt}{}{}{+}} \frac{a_{n+1}}{b_{n+1}} = \frac{b_{n+1}A_{n}+a_{n+1}A_{n-1}}{b_{n+1}B_{n}+a_{n+1}B_{n-1}}.
\end{align}  
The numerator polynomials of the corresponding continued fraction are $A_n$, and the denominator polynomials are $B_n$. Further, by the determinant formula, we obtain
\begin{align*}
\det\mathbb{F}_{n+1}= \begin{vmatrix}
		\mathcal{P}_{n+1}(z)	& -\mathcal{Q}_{n+1}(z)\\
		\mathcal{P}_{n}(z)	& -\mathcal{Q}_{n}(z)
	\end{vmatrix} = \prod_{j=1}^{n}\lambda_j (z-a_j),
\end{align*}
and hence, $\mathbb{F}_{n+1}$ is non-singular. In what follows, the product $\prod_{j=1}^{n}\lambda_j (z-a_j)$ will be denoted by the short-hand notation $\Lambda_n$.
These identities will be used to prove some of the results presented in this manuscript. \par
\begin{definition}\label{def homography mapping}
	The $\dot{=}$ notation used in \cite{Castillo chapter 2017} has been adapted for the homography mapping
	\begin{align*}
		r(z)=\dfrac{a(z)u(z)+b(z)}{c(z)u(z)+d(z)}, \quad a(z)d(z)-b(z)c(z) \neq 0.
	\end{align*}
	A pure rational spectral transformation is referred to as the transformation of a function $u(z)$ \cite{Castillo chapter 2017} given by
	\begin{align*}
		r(z) \dot{=} \mathbf{A}(z)u(z), \quad \mbox{where} \quad \mathbf{A}(z) = \begin{bmatrix}
			a(z)	&  b(z)\\
			c(z)	&  d(z)
		\end{bmatrix}.
	\end{align*} 
\end{definition}
%Throughout this section, we follow the analysis given in \cite{Paco perturbed recurrence 1990}. 
A spectral transformation changes the $m-$ function $m_\alpha$ \eqref{m_alpha} into a new $m-$ function $m_\beta$, associated with the measure $d\beta$, which is nothing more than a modification of the original measure $d\alpha$. We call the transformation of $m_\beta$ a pure rational spectral transformation when
\begin{align*}
	m_\beta \dot{=}  A(z) m_\alpha,
\end{align*}
where $a(z), b(z), c(z)$ and $d(z)$ are non-zero polynomials. \par
\subsection{Generalized co-modified $R_I$ polynomials}

\noindent
Let $\mathcal{R}_{I}(z;\mu_k,\nu_k)$ be the continued fraction related to the co-modified polynomial. Using \eqref{R1_k+1_z continued fraction}, $\mathcal{R}_{I}(z;\mu_k,\nu_k)$ can be written as
\begin{align}\label{R1_mu nu}
	\mathcal{R}_{I}(z;\mu_k,&\nu_k) = \frac{1}{z-c_0} \mathbin{\genfrac{}{}{0pt}{}{}{-}} \frac{\lambda_1(z-a_1)}{z-c_1} \mathbin{\genfrac{}{}{0pt}{}{}{-}}  \mathbin{\genfrac{}{}{0pt}{}{}{\cdots}} \mathbin{\genfrac{}{}{0pt}{}{}{-}} \frac{\nu_k \lambda_{k}(z-a_{k})}{z-c_{k}-\mu_k} \mathbin{\genfrac{}{}{0pt}{}{}{-}} \frac{\lambda_{k+1}(z-a_{k+1})}{z-c_{k+1}} \mathbin{\genfrac{}{}{0pt}{}{}{-}} \mathbin{\genfrac{}{}{0pt}{}{}{\cdots}} \nonumber \\
	%	&= \frac{1}{z-c_0-\dfrac{\lambda_1(z-a_1)}{z-c_1- \ldots-\dfrac{\nu_k \lambda_{k}(z-a_{k})}{z-c_{k}-\mu_k- \dfrac{\lambda_{k+1}(z-a_{k+1})}{z-c_{k+1}-\ldots}}}}\nonumber \\
	%	&= \frac{1}{z-c_0-\dfrac{\lambda_1(z-a_1)}{z-c_1- \ldots-\dfrac{\nu_k \lambda_{k}(z-a_{k})}{z-c_{k}-\mu_k -\lambda_{k+1}(z-a_{k+1})\mathcal{R}_{I}^{k+1}(z) }}} \nonumber\\
	&= \frac{1}{z-c_0} \mathbin{\genfrac{}{}{0pt}{}{}{-}} \frac{\lambda_1(z-a_1)}{z-c_1} \mathbin{\genfrac{}{}{0pt}{}{}{-}} \mathbin{\genfrac{}{}{0pt}{}{}{\cdots}} \mathbin{\genfrac{}{}{0pt}{}{}{-}} \frac{\nu_k \lambda_{k}(z-a_{k})}{z-c_{k}-\mu_k-\lambda_{k+1}(z-a_{k+1})\mathcal{R}_{I}^{k+1}(z)}.
\end{align}

\begin{theorem}\label{R1_mu_nu to R1 remark}
	$\mathcal{R}_{I}(z;\mu_k,\nu_k)$ defines a rational spectral transformation of $\mathcal{R}_{I}(z)$ as
	\begin{align*}
		\mathcal{R}_{I}(z;\mu_k,\nu_k) \dot{=} \mathbf{M}_k(z) \mathcal{R}_{I}(z),
	\end{align*}
	where $\mathbf{M}_k$ is the transfer matrix
	\begin{align*}	\mathbf{M}_k = \begin{bmatrix}
			\hat{\mathcal{S}}_k(z) \mathcal{P}_k(z) +\Lambda_k	 & -\mathcal{Q}_k(z)\hat{\mathcal{S}}_k(z)\\
			-\mathcal{S}_k(z)\mathcal{P}_k(z)	 & \Lambda_k+\mathcal{S}_k(z)\mathcal{Q}_k(z)
		\end{bmatrix},
	\end{align*}
	and where
	\begin{align*}
		\hat{\mathcal{S}}_k(z) &= -\mu_ k \mathcal{Q}_k(z)-(\nu_k-1)\lambda_k(z-a_k)\mathcal{Q}_{k-1}(z), \\
		\mathcal{S}_k(z) &=\mu_k \mathcal{P}_k(z)+(\nu_k-1)\lambda_k (z-a_k) \mathcal{P}_{k-1}(z).
	\end{align*}
\end{theorem}

\begin{proof}
First, we will find the relation between the continued fraction $\mathcal{R}_{I}(z;\mu_k,\nu_k)$ associated with the perturbations \eqref{co-recursive condition R1} and \eqref{co-dilated condition R1} and its $(k+1)$-th approximant $\mathcal{R}_{I}^{k+1}(z)$. The continued fraction expansion \eqref{R1_mu nu}, in comparison with \eqref{chihara identity}, gives
	\begin{align}
		&\mathcal{R}_{I}(z;\mu_k,\nu_k)  
		=\frac{(z-c_{k}-\mu_k-\lambda_{k+1}(z-a_{k+1})\mathcal{R}_{I}^{k+1}(z))\mathcal{Q}_{k}-\nu_k\lambda_{k}(z-a_{k})\mathcal{Q}_{k-1}}{(z-c_{k}-\mu_k-\lambda_{k+1}(z-a_{k+1})\mathcal{R}_{I}^{k+1}(z))\mathcal{P}_{k}-\nu_k \lambda_{k}(z-a_{k})\mathcal{P}_{k-1}} \nonumber\\
		%&= \frac{(z-c_{k})Q_{k}-\lambda_{k}(z-a_{k})Q_{k-1}-\mu_k Q_{k}-\lambda_{k+1}(z-a_{k+1})\mathcal{R}_{I}^{k+1}(z) Q_{k}-\nu_k \lambda_{k}(z-a_{k})Q_{k-1}+\lambda_{k}(z-a_{k})Q_{k-1}}{ (z-c_{k})P_{k}-\lambda_{k}(z-a_{k})P_{k-1}-\mu_k P_{k}-\lambda_{k+1}(z-a_{k+1})\mathcal{R}_{I}^{k+1}(z) P_{k}-\nu_k \lambda_{k}(z-a_{k})P_{k-1}+\lambda_{k}(z-a_{k})P_{k-1}} \\
		&= \frac{\mathcal{Q}_{k+1} -\lambda_{k+1}(z-a_{k+1})\mathcal{R}_{I}^{k+1}(z) \mathcal{Q}_{k}-\mu_k \mathcal{Q}_{k} +(1-\nu_k)\lambda_{k}(z-a_{k})\mathcal{Q}_{k-1}}{\mathcal{P}_{k+1} -\lambda_{k+1}(z-a_{k+1})\mathcal{R}_{I}^{k+1}(z) \mathcal{P}_{k} -\mu_k \mathcal{P}_{k}+(1-\nu_k)\lambda_{k}(z-a_{k})\mathcal{P}_{k-1}} \nonumber\\
		&= 	\frac{\mathcal{A}(z)\mathcal{R}_{I}^{k+1}(z)+\mathcal{B}(z)}{\mathcal{C}(z)\mathcal{R}_{I}^{k+1}(z)+\mathcal{D}(z)}, \label{R1_mu_nu to R1_k+1}
	\end{align}
	where
	\begin{align*}
		& \mathcal{A}(z) = \lambda_{k+1}(z-a_{k+1}) \mathcal{Q}_{k}, \quad \mathcal{B}(z) = -\mathcal{Q}_{k+1}+\mu_k \mathcal{Q}_{k} +(\nu_k-1)\lambda_{k}(z-a_{k})\mathcal{Q}_{k-1}, \nonumber \\
		& \mathcal{C}(z) = \lambda_{k+1}(z-a_{k+1}) \mathcal{P}_{k}, \quad \mathcal{D}(z) = -\mathcal{P}_{k+1}+\mu_k \mathcal{P}_{k} +(\nu_k-1)\lambda_{k}(z-a_{k})\mathcal{P}_{k-1}. \nonumber 
	\end{align*}
Now, looking at \Cref{def homography mapping}, it can be concluded that $\mathcal{R}_{I}(z;\mu_k,\nu_k)$ is the rational spectral transformation of its $(k+1)$-th approximant $\mathcal{R}_{I}^{k+1}(z)$. Further, putting $\mu_k = 0$ and $\nu_k = 1$ implies $\mathcal{R}_{I}(z)=\mathcal{R}_{I}(z;\mu_k,\nu_k)$, and formula \eqref{R1_mu_nu to R1_k+1} takes the form
\begin{align}
	&\mathcal{R}_{I}(z) = \frac{\mathcal{Q}_{k+1} -\lambda_{k+1}(z-a_{k+1})\mathcal{R}_{I}^{k+1}(z) \mathcal{Q}_{k}}{ \mathcal{P}_{k+1}  -\lambda_{k+1}(z-a_{k+1})\mathcal{R}_{I}^{k+1}(z) \mathcal{P}_{k}},\nonumber\\
	%		&	P_{k+1}\mathcal{R}_{I}(z)  -\lambda_{k+1}(z-a_{k+1})\mathcal{R}_{I}^{k+1}(z) P_{k}\mathcal{R}_{I}(z) = Q_{k+1} -\lambda_{k+1}(z-a_{k+1})\mathcal{R}_{I}^{k+1}(z)Q_{k} \\
	\Longrightarrow &\lambda_{k+1}(z-a_{k+1})\mathcal{R}_{I}^{k+1}(z) \mathcal{Q}_{k}-\lambda_{k+1}(z-a_{k+1})\mathcal{R}_{I}^{k+1}(z) \mathcal{P}_{k}\mathcal{R}_{I}(z) = \mathcal{Q}_{k+1}-\mathcal{P}_{k+1}\mathcal{R}_{I}(z), \nonumber\\
	\Longrightarrow &\lambda_{k+1}(z-a_{k+1})\mathcal{R}_{I}^{k+1}(z)[\mathcal{Q}_{k}-\mathcal{P}_{k}\mathcal{R}_{I}(z)] = \mathcal{Q}_{k+1}-\mathcal{P}_{k+1}\mathcal{R}_{I}(z), \nonumber\\
	\Longrightarrow &\lambda_{k+1}(z-a_{k+1})\mathcal{R}_{I}^{k+1}(z)= \frac{\mathcal{P}_{k+1}\mathcal{R}_{I}(z)-\mathcal{Q}_{k+1}}{\mathcal{P}_{k}\mathcal{R}_{I}(z)-\mathcal{Q}_{k}}. \label{R1_k+1 to R1_z}
\end{align}
Now, eliminating $\mathcal{R}_{I}^{k+1}(z)$ from \eqref{R1_mu_nu to R1_k+1} and \eqref{R1_k+1 to R1_z}, we obtain
	\begin{align*}
		&\mathcal{R}_{I}(z;\mu_k,\nu_k) = \frac{\mathcal{A}(z)\mathcal{R}_{I}^{k+1}(z)+\mathcal{B}(z)}{\mathcal{C}(z)\mathcal{R}_{I}^{k+1}(z)+\mathcal{D}(z)} \\
		&= \dfrac{\lambda_{k+1}(z-a_{k+1}) \mathcal{Q}_{k}\mathcal{R}_{I}^{k+1}(z)-\mathcal{Q}_{k+1}+\mu_k \mathcal{Q}_{k}+(\nu_k-1)\lambda_{k}(z-a_{k})\mathcal{Q}_{k-1} }{\lambda_{k+1}(z-a_{k+1}) \mathcal{P}_{k}\mathcal{R}_{I}^{k+1}(z)-\mathcal{P}_{k+1}+\mu_k \mathcal{P}_{k} +(\nu_k-1)\lambda_{k}(z-a_{k})\mathcal{P}_{k-1}} \\
		%&= \dfrac{\lambda_{k+1}(z-a_{k+1}) Q_{k}\left[\dfrac{1}{\lambda_{k+1}(z-a_{k+1})}\dfrac{P_{k+1}\mathcal{R}_{I}(z)-Q_{k+1}}{P_{k}\mathcal{R}_{I}(z)-Q_{k}}\right]-Q_{k+1}+\mu_k Q_{k} +(\nu_k-1)\lambda_{k}(z-a_{k})Q_{k-1}}{\lambda_{k+1}(z-a_{k+1}) P_{k}\left[\dfrac{1}{\lambda_{k+1}(z-a_{k+1})}\dfrac{P_{k+1}\mathcal{R}_{I}(z)-Q_{k+1}}{P_{k}\mathcal{R}_{I}(z)-Q_{k}}\right]-P_{k+1}+\mu_k P_{k} +(\nu_k-1)\lambda_{k}(z-a_{k})P_{k-1}} \\
		&= \frac{\mathcal{Q}_{k}\mathcal{P}_{k+1}\mathcal{R}_{I}(z)-\mathcal{Q}_{k}\mathcal{Q}_{k+1}-\mathcal{Q}_{k+1}[\mathcal{P}_{k}\mathcal{R}_{I}(z)-\mathcal{Q}_{k}]+\mu_k \mathcal{Q}_{k}[\mathcal{P}_{k}\mathcal{R}_{I}(z)-\mathcal{Q}_{k}] +}
		{\mathcal{P}_{k}\mathcal{P}_{k+1}\mathcal{R}_{I}(z)-\mathcal{Q}_{k+1}\mathcal{P}_{k}-\mathcal{P}_{k+1}[\mathcal{P}_{k}\mathcal{R}_{I}(z)-\mathcal{Q}_{k}]+\mu_k \mathcal{P}_{k}[\mathcal{P}_{k}\mathcal{R}_{I}(z)-\mathcal{Q}_{k}] +} \\
		& \hspace{8cm}\dfrac{+(\nu_k-1)\lambda_{k}(z-a_{k})\mathcal{Q}_{k-1}[\mathcal{P}_{k}\mathcal{R}_{I}(z)-\mathcal{Q}_{k}]}{+(\nu_k-1) \lambda_{k}(z-a_{k}) \mathcal{P}_{k-1}[\mathcal{P}_{k}\mathcal{R}_{I}(z)-\mathcal{Q}_{k}]} \\
		%&= \frac{[Q_{k}P_{k+1}-Q_{k+1}P_{k}+\mu_k Q_{k}P_{k}+ (\nu_k-1)\lambda_{k}(z-a_{k})Q_{k-1}P_{k}]\mathcal{R}_{I}(z)-\mu_k[Q_{k}]^2
		%	-(\nu_k-1)\lambda_{k}(z-a_{k})Q_{k-1}Q_k}
		%{[\mu_k[P_{k}]^2+(\nu_k-1)\lambda_{k}(z-a_{k})P_{k-1}P_{k}] \mathcal{R}_{I}(z) + P_{k+1}Q_{k}-Q_{k+1}P_{k}- \mu_k P_{k}Q_{k}-(\nu_k-1)\lambda_{k}(z-a_{k})P_{k-1}Q_{k} } \\
		&= \frac{[-\Lambda_k+\mu_k \mathcal{Q}_{k}\mathcal{P}_{k} +(\nu_k-1)\lambda_{k}(z-a_{k})\mathcal{Q}_{k-1}\mathcal{P}_{k}]\mathcal{R}_{I}(z)-\mu_k[\mathcal{Q}_{k}]^2-}
		{[\mu_k[\mathcal{P}_{k}]^2+(\nu_k-1)\lambda_{k}(z-a_{k})\mathcal{P}_{k-1}\mathcal{P}_{k}] \mathcal{R}_{I}(z) - \Lambda_k-\mu_k \mathcal{Q}_{k}\mathcal{P}_{k} - } \\
		& \hspace{9cm} \dfrac{-(\nu_k-1)\lambda_{k}(z-a_{k})\mathcal{Q}_{k-1}\mathcal{Q}_k}{-(\nu_k-1)\lambda_{k}(z-a_{k})\mathcal{P}_{k-1}\mathcal{Q}_{k}} \\
		&=\dfrac{(\mathbf{M}_k)_{11}\mathcal{R}_{I}(z)+(\mathbf{M}_k)_{12}}{(\mathbf{M}_k)_{21}\mathcal{R}_{I}(z)+ (\mathbf{M}_k)_{22}}, 
	\end{align*}
	where
	\begin{align*}
		&(\mathbf{M}_k)_{11} = \Lambda_k- (\nu_k-1)\lambda_{k}(z-a_{k})\mathcal{Q}_{k-1}\mathcal{P}_{k}-\mu_k \mathcal{Q}_{k}\mathcal{P}_{k},  \\
		&(\mathbf{M}_k)_{12}=(\nu_k-1)\lambda_{k}(z-a_{k})\mathcal{Q}_{k-1}\mathcal{Q}_{k}+\mu_k[\mathcal{Q}_{k}]^2, \\
		&(\mathbf{M}_k)_{21}= -(\nu_k-1)\lambda_{k}(z-a_{k})\mathcal{P}_{k-1}\mathcal{P}_{k}-\mu_k[\mathcal{P}_{k}]^2,\\
		& (\mathbf{M}_k)_{22}= \Lambda_k+(\nu_k-1)\lambda_{k}(z-a_{k})\mathcal{P}_{k-1}\mathcal{Q}_{k}+\mu_k \mathcal{Q}_{k}\mathcal{P}_{k}. 
	\end{align*} 
Now, in the view of \Cref{def homography mapping}, we obtain the required result. \qedhere 
\end{proof}

The following specific cases of generalized co-recursive and generalized co-dilated $R_I$ polynomials are immediate consequence of results proved earlier in this section.

\begin{remark}
	Let $\mathcal{R}_{I}(z;\mu_k)$ be the continued fraction corresponding to co-recursive perturbation \eqref{co-recursive condition R1}. With $\nu_k = 1$, relation \eqref{R1_mu_nu to R1_k+1} reduces to give a relation between $\mathcal{R}_{I}(z;\mu_k)$ and $\mathcal{R}_{I}^{k+1}(z)$.
	\begin{align*}
		\mathcal{R}_{I}(z;\mu_k) &\dot{=} \begin{bmatrix}
			\lambda_{k+1}(z-a_{k+1}) \mathcal{Q}_{k}	&  -\mathcal{Q}_{k+1}+\mu_k \mathcal{Q}_{k}\\
			\lambda_{k+1}(z-a_{k+1}) \mathcal{P}_{k} & -\mathcal{P}_{k+1}+\mu_k \mathcal{Q}_{k}
		\end{bmatrix} \mathcal{R}_{I}^{k+1}(z).
		%\mathcal{R}_{I}(z;\mu_k) &= \frac{\mathcal{A}(z)\mathcal{R}_{I}^{k+1}(z)+\mathcal{B}(z)}{\mathcal{C}(z)\mathcal{R}_{I}^{k+1}(z)+\mathcal{D}(z)}, \\
		%\label{R1_mu to R1_k+1 formula}
		%\text{where} \quad	 \mathcal{A}(z) &= \lambda_{k+1}(z-a_{k+1}) \mathcal{Q}_{k}, \qquad \mathcal{B}(z) = -\mathcal{Q}_{k+1} +\mu_k \mathcal{Q}_{k},  \\
		%\mathcal{C}(z) &= \lambda_{k+1}(z-a_{k+1}) \mathcal{P}_{k}, \qquad \mathcal{D}(z) = -\mathcal{P}_{k+1} +\mu_k \mathcal{P}_{k}. 
	\end{align*}
\end{remark}

%\begin{proof}
%	Using \eqref{s_mu}, \eqref{chihara identity} and some elementary techniques of continued fraction, we have the result.
%	\begin{align*}
%	\mathcal{R}_{I}(z;\mu_k)  
%	&=\frac{(z-c_{k}-\mu_k-\lambda_{k+1}(z-a_{k+1})\mathcal{R}_{I}^{k+1}(z))Q_{k}-\lambda_{k}(z-a_{k})Q_{k-1}}{(z-c_{k}-\mu_k-\lambda_{k+1}(z-a_{k+1})\mathcal{R}_{I}^{k+1}(z))P_{k}-\lambda_{k}(z-a_{k})P_{k-1}} \nonumber\\
%	&= \frac{(z-c_{k})Q_{k}-\lambda_{k}(z-a_{k})Q_{k-1} -\mu_k Q_{k}-\lambda_{k+1}(z-a_{k+1})\mathcal{R}_{I}^{k+1}(z) Q_{k}}{ (z-c_{k})P_{k}-\lambda_{k}(z-a_{k})P_{k-1} -\mu_k P_{k}-\lambda_{k+1}(z-a_{k+1})\mathcal{R}_{I}^{k+1}(z) P_{k}} \nonumber\\
%	&= \frac{Q_{k+1} -\mu_k Q_{k}-\lambda_{k+1}(z-a_{k+1})\mathcal{R}_{I}^{k+1}(z) Q_{k}}{ P_{k+1} -\mu_k P_{k}-\lambda_{k+1}(z-a_{k+1})\mathcal{R}_{I}^{k+1}(z) P_{k}} \nonumber\\
%	&= 	\frac{A(z)\mathcal{R}_{I}^{k+1}(z)+B(z)}{C(z)\mathcal{R}_{I}^{k+1}(z)+D(z)} \nonumber	
%	\end{align*}
%\end{proof}

\begin{remark}
	A relation between continued fraction associated with co-recursion \eqref{co-recursive condition R1} and its $(k+1)$-th approximant can be given in the following way:
	\begin{align*}
		\mathcal{R}_{I}(z;\mu_k) &\dot{=} \begin{bmatrix}
			\Lambda_k+\mu_k \mathcal{Q}_{k}\mathcal{P}_{k}	&  -\mu_k \mathcal{Q}_{k}^2\\
			\mu_k \mathcal{P}_{k}^2 & \Lambda_k-\mu_k \mathcal{P}_{k}\mathcal{Q}_{k}
		\end{bmatrix} \mathcal{R}_{I}(z).
		%\mathcal{R}_{I}(z;\mu_k) &= \dfrac{\mathcal{A}_\mu(z)\mathcal{R}_{I}(z)+\mathcal{B}_\mu(z)}{\mathcal{C}_\mu(z)\mathcal{R}_{I}(z)+\mathcal{D}_\mu(z)}, &\\
		%\mbox{where} \quad	\mathcal{A}_\mu(z) = \displaystyle\prod_{j=1}^{k}\lambda_j (z-a_j)+\mu_k \mathcal{Q}_{k}\mathcal{P}_{k}, &\qquad \mathcal{B}_\mu(z)=-\mu_k [\mathcal{Q}_{k}]^2, \\
		%\mathcal{D}_\mu(z)= \displaystyle\prod_{j=1}^{k}\lambda_j (z-a_j)-\mu_k \mathcal{P}_{k}\mathcal{Q}_{k}, &\qquad \mathcal{C}_\mu(z)= \mu_k[ \mathcal{P}_{k}]^2.
	\end{align*}
\end{remark}
\begin{remark}
	Let $\mathcal{R}_{I}(z;\nu_k)$ be the continued fraction associated with co-dilation \eqref{co-dilated condition R1}. With $\mu_k = 0$, equation \eqref{R1_mu_nu to R1_k+1} reduces to give a relation between $\mathcal{R}_{I}(z;\nu_k)$ and its $(k+1)$-th approximant.
	\begin{align*}
		\mathcal{R}_{I}(z;\nu_k) &\dot{=} \begin{bmatrix}
			\lambda_{k+1}(z-a_{k+1}) \mathcal{Q}_{k}	&  -\mathcal{Q}_{k+1}+(\nu_k-1)\lambda_{k}(z-a_{k})\mathcal{Q}_{k-1}\\
			\lambda_{k+1}(z-a_{k+1}) \mathcal{P}_{k} & -\mathcal{P}_{k+1} +(\nu_k-1)\lambda_{k}(z-a_{k})\mathcal{P}_{k-1}
		\end{bmatrix} \mathcal{R}_{I}^{k+1}(z).
		%&\mathcal{R}_{I}(z;\nu_k)=\frac{\mathcal{A}(z)\mathcal{R}_{I}^{k+1}(z)+\mathcal{B}(z)}{\mathcal{C}(z)\mathcal{R}_{I}^{k+1}(z)+\mathcal{D}(z)}, \\
		%	\text{where} \quad	& \mathcal{A}(z) = \lambda_{k+1}(z-a_{k+1}) \mathcal{Q}_{k}, \qquad \mathcal{B}(z) = -\mathcal{Q}_{k+1} +(\nu_k-1)\lambda_{k}(z-a_{k})\mathcal{Q}_{k-1}, \nonumber\\
		%	& \mathcal{C}(z) = \lambda_{k+1}(z-a_{k+1}) \mathcal{P}_{k}, \qquad \mathcal{D}(z) = -\mathcal{P}_{k+1} +(\nu_k-1)\lambda_{k}(z-a_{k})\mathcal{P}_{k-1}. \nonumber
	\end{align*}
\end{remark}
%\begin{proof}
%Equation \eqref{S_nu} with \eqref{chihara identity} gives
%\begin{align*}
%	\mathcal{R}_{I}(z;\nu_k)
%	&=\frac{(z-c_{k}-\lambda_{k+1}(z-a_{k+1})\mathcal{R}_{I}^{k+1}(z))Q_{k}-\nu_k\lambda_{k}(z-a_{k})Q_{k-1}}{(z-c_{k}-\lambda_{k+1}(z-a_{k+1})\mathcal{R}_{I}^{k+1}(z))P_{k}-\nu_k \lambda_{k}(z-a_{k})P_{k-1}} \\
%	&= \frac{(z-c_{k})Q_{k}-\lambda_{k}(z-a_{k})Q_{k-1}-\lambda_{k+1}(z-a_{k+1})\mathcal{R}_{I}^{k+1}(z) Q_{k}-\nu_k \lambda_{k}(z-a_{k})Q_{k-1}+\lambda_{k}(z-a_{k})Q_{k-1}}{ (z-c_{k})P_{k}-\lambda_{k}(z-a_{k})P_{k-1}-\lambda_{k+1}(z-a_{k+1})\mathcal{R}_{I}^{k+1}(z) P_{k}-\nu_k \lambda_{k}(z-a_{k})P_{k-1}+\lambda_{k}(z-a_{k})P_{k-1}} \\
%	&= \frac{Q_{k+1} -\lambda_{k+1}(z-a_{k+1})\mathcal{R}_{I}^{k+1}(z) Q_{k}+(1-\nu_k)\lambda_{k}(z-a_{k})Q_{k-1}}{P_{k+1} -\lambda_{k+1}(z-a_{k+1})\mathcal{R}_{I}^{k+1}(z) P_{k}+(1-\nu_k)\lambda_{k}(z-a_{k})P_{k-1}} \\
%	&= 	\frac{A(z)\mathcal{R}_{I}^{k+1}(z)+B(z)}{C(z)\mathcal{R}_{I}^{k+1}(z)+D(z)} 
%\end{align*}	
%\end{proof}
%A result similar to Theorem \ref{R1_mu to R1_z} can be proved for co-dilated case as well in the following way:
\begin{remark}
	A relation between continued fraction associated with co-dilation \eqref{co-dilated condition R1} and its $(k+1)$-th approximant can be given in the following way:
	\begin{align*}
		\mathcal{R}_{I}(z;\nu_k) \dot{=}\begin{bmatrix}
			\Lambda_k+(\nu_k-1)\lambda_{k}(z-a_{k})\mathcal{Q}_{k-1}\mathcal{P}_{k}	&  -(\nu_k-1)\lambda_{k}(z-a_{k})\mathcal{Q}_{k-1}\mathcal{Q}_{k}\\
			(\nu_k-1)\lambda_{k}(z-a_{k})\mathcal{P}_{k-1}\mathcal{P}_{k} & \Lambda_k-(\nu_k-1)\lambda_{k}(z-a_{k})\mathcal{P}_{k-1}\mathcal{Q}_{k}
		\end{bmatrix} \mathcal{R}_{I}(z).
		%&\mathcal{R}_{I}(z;\nu_k) = \dfrac{\mathcal{A}_\mu(z)\mathcal{R}_{I}(z)+\mathcal{B}_\mu(z)}{\mathcal{C}_\mu(z)\mathcal{R}_{I}(z)+\mathcal{D}_\mu(z)}, \\
		%\text{where} \qquad	&\mathcal{A}_\nu(z) = \displaystyle\prod_{j=1}^{k}\lambda_j (z-a_j)+ (\nu_k-1)\lambda_{k}(z-a_{k})\mathcal{Q}_{k-1}\mathcal{P}_{k},  \\
		%&	\mathcal{B}_\nu(z)=-(\nu_k-1)\lambda_{k}(z-a_{k})\mathcal{Q}_{k-1}\mathcal{Q}_{k}, \\	&\mathcal{D}_\nu= \displaystyle\prod_{j=1}^{k}\lambda_j (z-a_j)-(\nu_k-1)\lambda_{k}(z-a_{k})\mathcal{P}_{k-1}\mathcal{Q}_{k}, \\ 
		%&   \mathcal{C}_\nu(z)= (\nu_k-1)\lambda_{k}(z-a_{k})\mathcal{P}_{k-1}\mathcal{P}_{k}.
	\end{align*}
\end{remark}
For additional details regarding rational spectral transformations, see \cite{Zhedanov Rational spectral 1997}. \par 
%\begin{remark}
%Following expression gives a relation between $\mathcal{R}_{I}(z;\mu_k,\nu_k)$ and $\mathcal{R}_{I}(z)$.
%\begin{align*}
%&	\mathcal{R}_{I}(z;\mu_k,\nu_k) = \dfrac{\mathcal{A}_{\mu,\nu}(z)\mathcal{R}_{I}(z)+\mathcal{B}_{\mu,\nu}(z)}{\mathcal{C}_{\mu,\nu}(z)\mathcal{R}_{I}(z)+\mathcal{D}_{\mu,\nu}(z)},\\
%\end{align*}
%\end{remark}
The next important aspect is the computation of perturbed polynomials from the original ones. This can be achieved in two ways: 
\begin{enumerate}
	\item The transfer matrix method
	\item The classical method.
\end{enumerate}
\subsection{The transfer matrix method}
The following result is important both from a computational and theoretical standpoint because $\mathbf{M}_k$ being a matrix, it is easy to compute the perturbed polynomials $\mathcal{P}_{n+1}(z;\mu_k,\nu_k)$ and $\mathcal{Q}_{n+1}(z;\mu_k,\nu_k)$.

\begin{theorem}\label{transfer matrix theorem}
	The following relation holds in $\mathbb{C}$:
	\begin{align*}
		\Lambda_k \begin{bmatrix}
			\mathcal{P}_{n+1}(z;\mu_k,\nu_{k}) & \mathcal{P}_{n}(z;\mu_k,\nu_{k})	\\
			-\mathcal{Q}_{n+1}(z;\mu_k,\nu_{k}) & -\mathcal{Q}_{n}(z;\mu_k,\nu_{k})
		\end{bmatrix} & = ~ cof(\mathbf{M}_k)(z) \begin{bmatrix}
		\mathcal{P}_{n+1}(z) &	\mathcal{P}_{n}(z)\\
		-\mathcal{Q}_{n+1}(z) & -\mathcal{Q}_{n}(z)
	\end{bmatrix},
	\end{align*}
where $cof(\cdot)$ is the cofactor matrix operator.
\end{theorem}
\begin{proof}
Let $\mathbb{F}_{n+1}(\mu_k,\nu_k)$ be the polynomial matrix corresponding to co-modified polynomials of $R_{I}$ type. From \eqref{F_n+1 to T_0 R1 pp}, we get
\begin{align*}
\mathbb{F}_{n+1}(\mu_k,\nu_k) &= \mathbf{T}_n \ldots \mathbf{T}_{k+1}\mathbf{T}_k (\mu_k,\nu_k) \mathbf{T}_{k-1} \ldots \mathbf{T}_0 \nonumber   \\
&=\mathbf{T}_n \ldots \mathbf{T}_{k+1}\mathbf{T}_k \mathbf{T}_{k-1} \ldots \mathbf{T}_0 (\mathbb{F}_{k+1} )^{-1}\mathbf{T}_k (\mu_k,\nu_k)\mathbb{F}_{k} \nonumber \\
&= \mathbb{F}_{n+1}(\mathbb{F}_{k+1} )^{-1} \mathbb{F}_{k+1}(\mu_k,\nu_k).
\end{align*}
This implies
\begin{align}
	\mathbb{F}^T_{n+1}(\mu_k,\nu_k) &=[\mathbb{F}_{k+1}(\mu_k,\nu_k)]^T (\mathbb{F}_{k+1} )^{-T} \mathbb{F}^T_{n+1} \label{F_k+1 transpose 1 R1 pp}.
\end{align}
Since
\begin{align}\label{F_k+1 transpose 2 R1 pp}
	[\mathbb{F}_{k+1}(\mu_k,\nu_k)]^T= \begin{bmatrix}
		\mathcal{P}_{k+1}(z;\mu_k,\nu_k)	& \mathcal{P}_{k}(z)\\
		-\mathcal{Q}_{k+1}(z;\mu_k,\nu_k)	& -\mathcal{Q}_{k}(z)
	\end{bmatrix} 
\end{align}
and by the determinant formula, $\det(\mathbb{F}_{k+1}) = \Lambda_k$,
\begin{align}
	(\mathbb{F}_{k+1})^{-T} &= \frac{1}{\Lambda_k} \begin{bmatrix}
		-\mathcal{Q}_{k}(z)	& -\mathcal{P}_{k}(z)\\
		\mathcal{Q}_{k+1}(z)	& \mathcal{P}_{k+1}(z)
	\end{bmatrix}. \label{T_k F_k tranpose inverse}
\end{align}
Using \eqref{F_k+1 transpose 2 R1 pp} and \eqref{T_k F_k tranpose inverse} in \eqref{F_k+1 transpose 1 R1 pp}, we obtain\\ 
$ \displaystyle
[\mathbb{F}_{k+1}(\mu_k,\nu_k)]^T (\mathbb{F}_{k+1} )^{-T}
$
\begin{align}\label{product matrix R2 pp}
&=\frac{1}{\Lambda_k}\begin{bmatrix}
	-\mathcal{P}_{k+1}(z;\mu_k,\nu_k) \mathcal{Q}_{k}(z)+\mathcal{P}_{k}(z)\mathcal{Q}_{k+1}(z)	& -\mathcal{P}_{k+1}(z;\mu_k,\nu_k)\mathcal{P}_{k}(z)+\mathcal{P}_{k}(z)\mathcal{P}_{k+1}(z) \\
	\mathcal{Q}_{k+1}(z;\mu_k,\nu_k)\mathcal{Q}_{k}(z)-\mathcal{Q}_{k}(z)\mathcal{Q}_{k+1}(z)	& \mathcal{Q}_{k+1}(z;\mu_k,\nu_k)\mathcal{P}_{k}(z)-\mathcal{Q}_{k}(z)\mathcal{P}_{k+1}(z)
\end{bmatrix} \nonumber\\
&= \frac{1}{\Lambda_k} \begin{bmatrix}
	\Lambda_k +\mathcal{S}_k(z)\mathcal{Q}_k(z) & \mathcal{S}_k(z)\mathcal{P}_k(z)	\\
	\mathcal{Q}_k(z)\hat{\mathcal{S}}_k(z) & \hat{\mathcal{S}}_k(z) \mathcal{P}_k(z) +\Lambda_k
\end{bmatrix} \nonumber\\
&=\frac{1}{\Lambda_k} \begin{bmatrix}
	(\mathbf{M}_k)_{22}	& -(\mathbf{M}_k)_{21}\\
	-(\mathbf{M}_k)_{12}	& (\mathbf{M}_k)_{11} 
\end{bmatrix}= \frac{cof(\mathbf{M}_k)(z)}{\Lambda_k}
\end{align}
The theorem follows after substituting \eqref{product matrix R2 pp} in \eqref{F_k+1 transpose 1 R1 pp}.
%\begin{align*}
%\dfrac{\mathcal{Q}_n(z;\mu_k,\nu_k)}{\mathcal{P}_n(z;\mu_k,\nu_k)}&= \dfrac{(\mathbf{M}_k)_{11}\dfrac{\mathcal{Q}_n(z)}{\mathcal{P}_n(z)}+(\mathbf{M}_k)_{12}}{(\mathbf{M}_k)_{21}\dfrac{\mathcal{Q}_n(z)}{\mathcal{P}_n(z)}+(\mathbf{M}_k)_{22}}.  
%\end{align*}
%Using \eqref{R1 continued fraction} and \Cref{def homography mapping}, the result follows.
%A clever observation between \Cref{transfer matrix theorem} and \Cref{R1_mu_nu to R1 remark} implies the result.
\end{proof}

The former theorem enables the study of the polynomials when finite number of recurrence coefficients are perturbed. In this regard, the following can be proved using techniques given in \cite{Castillo perturbed szego 2014, swami vinay R2 2022}.

\begin{theorem}
Let $k,m$ be two fixed non-negative integers with $m < k$. Then for $n > m$, the following relation holds:
\begin{align*}
\prod_{j=m}^{k}\Lambda_j \mathbb{F}_{n+1}(\mu_m,\nu_m,\ldots \mu_k,\nu_k)  & = \prod_{j=m}^{k}cof(\mathbf{M}_j)(z) \mathbb{F}_{n+1}.
\end{align*}
\end{theorem}

%\begin{theorem}
%	Let $\mathcal{R}_{I}(z;\mu_k,\nu_k)$  be the continued fraction associated with the perturbations \eqref{co-recursive condition R1} and \eqref{co-dilated condition R1}. Then $\mathcal{R}_{I}(z;\mu_k,\nu_k)$ is a pure rational spectral transformation of $\mathcal{R}_{I}(z)$ given by
%	\begin{align*}
%		\mathcal{R}_{I}(z;\mu_k,\nu_k) \dot{=} cof(\mathbf{M}_k)\mathcal{R}_{I}(z),
%	\end{align*}
%	where cof(.) is the cofactor matrix operator.
%\end{theorem}

%\begin{corollary}
%	$\mathcal{R}_{I}(z;\mu_m,\nu_m,\ldots \mu_k,\nu_k)$ is a pure rational spectral transformation of $\mathcal{R}_{I}(z)$ given by
%	\begin{align*}
%		\mathcal{R}_{I}(z;\mu_m,\nu_m,\ldots \mu_k,\nu_k) = cof \left(\prod_{j=m}^{k}\mathbf{M}_j \right)\mathcal{R}_{I}(z).
%	\end{align*}
%\end{corollary}
\subsection{The classical method}
A disadvantage of \Cref{Theorem s_k_x R1 pp} is that it requires the knowledge of associated polynomials, $\mathcal{P}^{(k+1)}_{n-k}(z)$ and $\mathcal{Q}^{(k+1)}_{n-k}(z)$, while \Cref{transfer matrix theorem} utilizes already given polynomials $\mathcal{P}_{n}(z)$ and $\mathcal{Q}_{n}(z)$ to compute the perturbed polynomials. 
\begin{theorem}\label{Theorem s_k_x R1 pp}
	The following formulas hold:
	\begin{align*}
		\mathcal{P}_{n+1}(z;\mu_k,\nu_k)&=\mathcal{P}_{n+1}(z), \qquad \mbox{$n < k$}, \\
		{\mathcal{P}}_{n+1}(z;\mu_k,\nu_k) &= \mathcal{P}_{n+1}(z) - \mathcal{S}_k(z)\mathcal{P}^{(k+1)}_{n-k}(z), \qquad \mbox{$n \geq k$}, \\
		\text{and} \quad \mathcal{Q}_{n+1}(z;\mu_k,\nu_k)&=\mathcal{Q}_{n+1}(z), \qquad \mbox{$n < k$}, \\
		{\mathcal{Q}}_{n+1}(z;\mu_k,\nu_k) &= \mathcal{Q}_{n+1}(z) + \hat{\mathcal{S}}_k(z)\mathcal{Q}^{(k+1)}_{n-k}(z), \qquad \mbox{$n \geq k$},
	\end{align*}
%	where $\mathcal{S}_k(x)=\mu_k \mathcal{P}_k(x)+(\nu_k-1)\lambda_k x \mathcal{P}_{k-1}(x)$ and $\hat{\mathcal{S}}_k(x)=-\mu_k \mathcal{Q}_k(x)-(\nu_k-1)\lambda_k x \mathcal{Q}_{k-1}(x)$.
\end{theorem}
\begin{proof}
	Any solution to \eqref{Generalised comodified equations} will be a linear combination of two linearly independent solutions, as is clear from the theory of difference equations. From \eqref{P_n+1 mu_k nu_k R1 pp},
	\begin{align*}
		\begin{bmatrix}
			\mathcal{P}_{n+1}(z;\mu_k,\nu_k) & -\mathcal{Q}_{n+1}(z;\mu_k,\nu_k)\\
			{\mathcal{P}}_{n}(z;\mu_k,\nu_k) & -\mathcal{Q}_{n}(z;\mu_k,\nu_k)
		\end{bmatrix}&= (\mathbf{T}_n \ldots \mathbf{T}_{k+1})\mathbf{T}_k(\mu_k,\nu_k) \begin{bmatrix}
			\mathcal{P}_{k}(z) & -\mathcal{Q}_{k}(z)\\
			{\mathcal{P}}_{k-1}(z) & -\mathcal{Q}_{k-1}(z)
		\end{bmatrix}.
	\end{align*}
	Use of \eqref{T_k + N_k} implies
	\begin{align*}
		&\begin{bmatrix}
			\mathcal{P}_{n+1}(z;\mu_k,\nu_k) & -\mathcal{Q}_{n+1}(z;\mu_k,\nu_k)\\
			{\mathcal{P}}_{n}(z;\mu_k,\nu_k) & -\mathcal{Q}_{n}(z;\mu_k,\nu_k)
		\end{bmatrix}= (\mathbf{T}_n \ldots \mathbf{T}_{k+1})(\mathbf{T}_k+\mathbf{N}_k) \begin{bmatrix}
			\mathcal{P}_{k}(z) & -\mathcal{Q}_{k}(z)\\
			{\mathcal{P}}_{k-1}(z) & -\mathcal{Q}_{k-1}(z)
		\end{bmatrix} \\
		 =& \begin{bmatrix}
			\mathcal{P}_{n+1}(z) & -\mathcal{Q}_{n+1}(z)\\
			{\mathcal{P}}_{n}(z) & -\mathcal{Q}_{n}(z)
		\end{bmatrix}+ \mathbf{T}_n \ldots \mathbf{T}_{k+1} \begin{bmatrix}
			-\mu_k & -(\nu_k-1) \lambda_k (z-a_k) \\
			0 & 0	
		\end{bmatrix} \begin{bmatrix}
			\mathcal{P}_{k}(z) & -\mathcal{Q}_{k}(z)\\
			{\mathcal{P}}_{k-1}(z) & -\mathcal{Q}_{k-1}(z)
		\end{bmatrix}\\
	=& \begin{bmatrix}
		\mathcal{P}_{n+1}(z) & -\mathcal{Q}_{n+1}(z)\\
		{\mathcal{P}}_{n}(z) & -\mathcal{Q}_{n}(z)
	\end{bmatrix}+ \mathbf{T}_n \ldots \mathbf{T}_{k+1} \begin{bmatrix}
		-\mathcal{S}_k(z) & -\hat{\mathcal{S}}_k(z)\\
		0 & 0
	\end{bmatrix}\\
		&=\begin{bmatrix}
			\mathcal{P}_{n+1}(z) & -\mathcal{Q}_{n+1}(z)\\
			{\mathcal{P}}_{n}(z) & -\mathcal{Q}_{n}(z)
		\end{bmatrix} + \begin{bmatrix}
			-\mathcal{S}_k(z) \mathcal{P}^{(k+1)}_{n-k}(z) & -\hat{\mathcal{S}}_k(z)\mathcal{Q}^{(k+1)}_{n-k}(z)\\
			-\mathcal{S}_k(z) \mathcal{P}^{(k+1)}_{n-k-1}(z) & -\hat{\mathcal{S}}_k(z)\mathcal{Q}^{(k+1)}_{n-k-1}(z)
		\end{bmatrix}
	\end{align*}
	which proves the theorem.
\end{proof}
Next, we compare the computational complexity between both approaches to compute perturbed $\mathcal{R}_{I}$ polynomials. Further, a comparative study of complexity involved in the computation of perturbed $\mathcal{R}_{I}$ polynomials (discussed in this work) and perturbed $\mathcal{R}_{II}$ polynomials (discussed in \cite{swami vinay R2 2022}) from both approaches is also provided.
\subsection{Comparison of computational cost between the classical and the transfer matrix method} In this section, we provide details for $\mathcal{P}_{4}(x;\mu_0,\nu_0,\mu_1,\nu_1)$ by explicitly counting the multiplication operation (abbreviated as \lq prod' now onwards) involved. Consider two consecutive perturbations at $k=0$ and $k=1$ in \eqref{R1}. In the Transfer matrix method, the following steps are involved:
\begin{enumerate}
	\item Computation of entries of transfer matrices $\mathbf{M}_{0}$ and $\mathbf{M}_{1}$ and product of these two matrices.
	\item Multiplication of the first row of the resultant matrix with the column vector $[\mathcal{P}_{4}(x),-\mathcal{Q}_{4}(x)]^T$.	
\end{enumerate}
Because the multiplication operation is considered to be the most expensive, we only consider counting distinct \lq prod' among real coefficients. The transfer matrices $\mathbf{M}_{0}$ and $\mathbf{M}_{1}$ are
\begin{align*}
	\mathbf{M}_{0}= \begin{bmatrix}
		0	& \mu_0\\
		0	& 0
	\end{bmatrix}, \quad \mathbf{M}_{1} =\begin{bmatrix}
		a_{11}	& a_{12}\\
		a_{21}	&  a_{22}
	\end{bmatrix},
\end{align*}
where $a_{11}=\mu_1(x-c_0)+\nu_1\lambda_1(x-a_1)$, $a_{12} = \mu_1(x-c_0)^2+(\nu_1-1)\lambda_1(x-a_1)(x-c_0)$, $a_{21}=-\mu_1$ and $a_{22}=-\mu_1(x-c_0)+\lambda_1(x-a_1)$. In this case, $a_{11}$ involves three distinct \lq prod', namely $\mu_1c_0$ and $\nu_1\lambda_1$. Similarly, $a_{12}$, $a_{21}$, and $a_{22}$ each have $5$, $0$, and $0$ distinct \lq prod', respectively (because $\mu_1c_0$ is already computed in $a_{11}$, it will not be counted again). The product $\mathbf{M}_{0}\mathbf{M}_{1}$
\begin{align*}
	\mathbf{M}_{0}\mathbf{M}_{1}= \begin{bmatrix}
		-\mu_0\mu_1	& \mu_0[\mu_1(x-c_0)+\lambda_1(x-a_1)]\\
		0	& 0
	\end{bmatrix}=\begin{bmatrix}
		b_{11}	& b_{12}\\
		b_{21}	&  b_{22}
	\end{bmatrix},
\end{align*}
requires 4 \lq prod'. As $\mathcal{P}_{4}(x)$ and $\mathcal{Q}_{4}(x)$ are (known) polynomials of degrees $4$ and $3$, $b_{11} \times \mathcal{P}_{4}(x)$ requires 5 \lq prod', and $b_{12} \times \mathcal{Q}_{4}(x)$ requires 8 \lq prod' (since $b_{12}$ in itself is a polynomial of degree 1). As a result, the computation of $\mathcal{P}_{4}(x;\mu_0,\nu_0,\mu_1,\nu_1)$ requires a total of 25 \lq prod'. \par
On the other hand, in the classical method, we have
\begin{align*}
	&\mathcal{P}_{0}(x)=1, \\
	&\mathcal{P}_{1}(x;\mu_0,\nu_0)=(x-c_{0}-\mu_0)\mathcal{P}_0(x), \\
	&\mathcal{P}_{2}(x;\mu_0,\nu_0,\mu_1,\nu_1)=(x-c_{1}-\mu_1)\mathcal{P}_1(x;\mu_0,\nu_0)-\nu_1\lambda_{1} (x^2+1)\mathcal{P}_{0}(x),\\
	&\mathcal{P}_{3}(x;\mu_0,\nu_0,\mu_1,\nu_1)=(x-c_{2})\mathcal{P}_2(x;\mu_0,\nu_0,\mu_1,\nu_1)-\lambda_{2} (x^2+1)\mathcal{P}_{1}(x;\mu_0,\nu_0), \\
	&\mathcal{P}_{4}(x;\mu_0,\nu_0,\mu_1,\nu_1)=(x-c_{3})\mathcal{P}_3(x;\mu_0,\nu_0,\mu_1,\nu_1)-\lambda_{3} (x^2+1)\mathcal{P}_{2}(x;\mu_0,\nu_0,\mu_1,\nu_1).
\end{align*}
Substitution can now be used to compute $\mathcal{P}_{4}(x;\mu_0,\nu_0,\mu_1,\nu_1)$. In the computation, 42 distinct \lq prod' are involved. In this case, the Transfer matrix method outperforms the classical method in terms of computational speed. Furthermore, we have compared the methods for higher degree polynomials (see Conclusion 2) and discovered that the difference in the number of \lq prod' between these two approaches grows even faster.

\subsubsection{Conclusions}
\begin{enumerate}
\item Both approaches have been shown to be equally efficient for generating lower-degree polynomials with a single perturbation.
\item According to \Cref{Fig 1}, the classical method requires less computation for $\mathcal{P}_{2}(x;\mu_0,\break \nu_0,\mu_1,\nu_1)$ and $\mathcal{P}_{3}(x;\mu_0,\nu_0,\mu_1,\nu_1)$. However, the number of \lq prod' involved in the computation of higher degree perturbed polynomials (with two consecutive perturbations) by the classical method (blue dashed line) is significantly greater than that involved in Transfer matrix method (solid red line). Furthermore, the growth of \lq prod' involved in the classical method is non-linear.
\item The cost of computing perturbed perturbed $\mathcal{R}_{I}$ (solid red line) and $\mathcal{R}_{II}$ (dash-dot magenta line) polynomials (see \cite{swami vinay R2 2022}) from the transfer matrix method is compared in \Cref{Fig 1}. Observe that, upto degree 4, lesser \lq prod' are required in the computation of perturbed $\mathcal{R}_{II}$ polynomials than in perturbed $\mathcal{R}_{I}$ polynomials, but after degree 5, the situation reverses, i.e., requiring more \lq prod' in the computation of perturbed $\mathcal{R}_{I}$ polynomials than in perturbed $\mathcal{R}_{II}$ polynomials.
\item It is interesting to note that (see \Cref{Fig 1}) the complexity of the classical method for computing perturbed $\mathcal{R}_{I}$ (dashed blue line) is higher than that of the perturbed $\mathcal{R}_{II}$ (dotted brown line) polynomials (see \cite{swami vinay R2 2022}).
\end{enumerate}
\begin{table}
	\caption{}\label{T6}
	\renewcommand{\arraystretch}{1.03}
	\centering
\begin{tabular}[b]{|p{2cm}|p{2.2cm}|p{2.2cm}|p{2cm}|p{2.2cm}|}
	\hline
	Degree of $\mathcal{P}_{n}(x)$& Classical Method $\mathcal{R}_{I}$ & Classical Method $\mathcal{R}_{II}$ &Transfer Matrix $\mathcal{R}_{I}$ &Transfer Matrix $\mathcal{R}_{II}$ \\
	\hline
	2 & 6  & 5 & 21  & 16\\
	\hline
	3 & 15 & 11 & 22  & 20\\
	\hline
	4 & 42 & 31 & 25  & 24\\
	\hline
	5 & 125 & 78 & 28  & 28 \\
	\hline
\end{tabular}
\end{table}
\begin{figure}
	\begin{center}
		\includegraphics[scale=0.3]{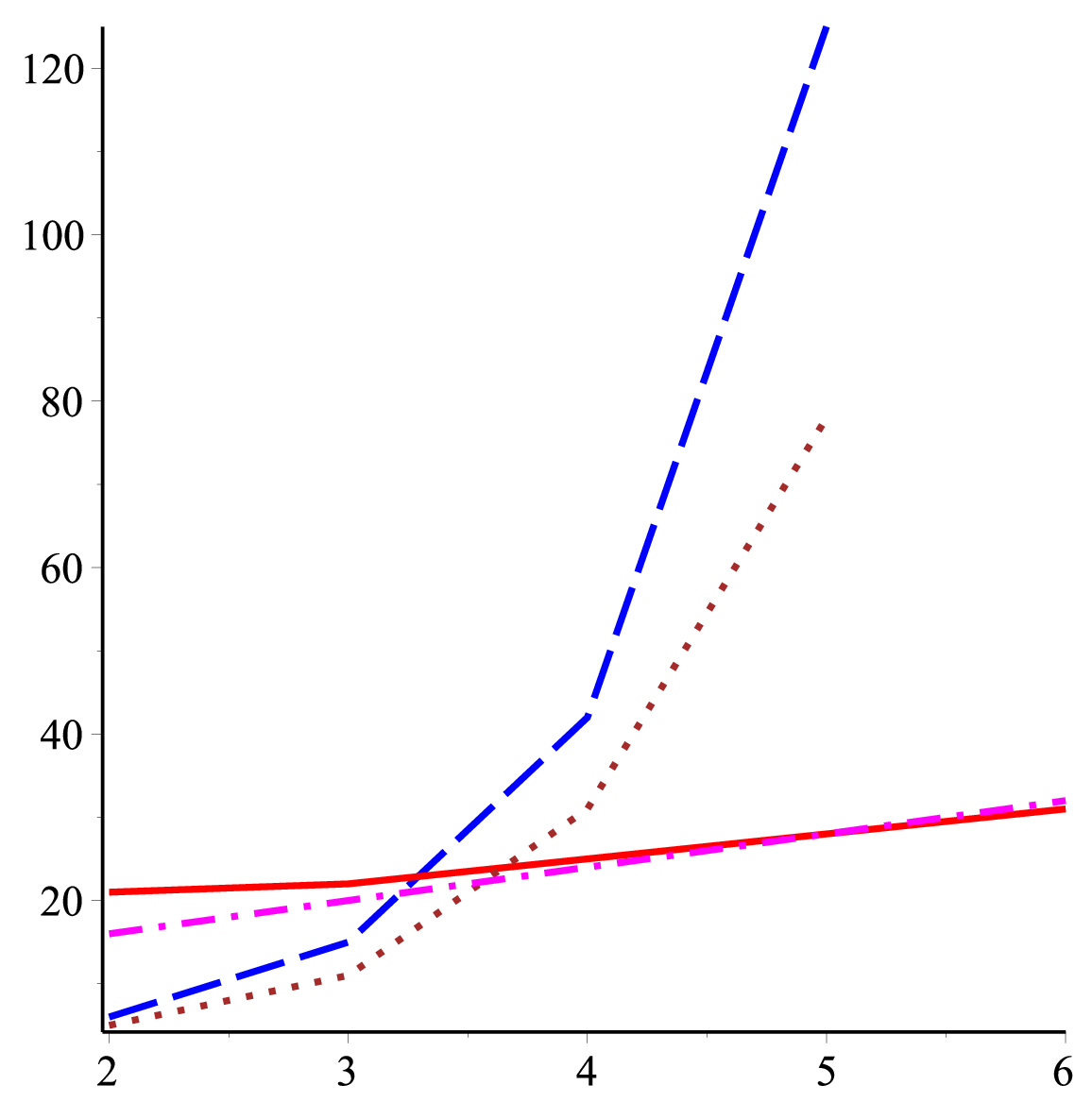}
		\caption{}
		\label{Fig 1}
	\end{center}
\end{figure}

\subsection{Inclusion, interlacing and monotonicity properties of zeros}\label{Inclusion interlacing and monotonicity}

To discuss the zeros, we consider a special form of $\mathcal{R}_{I}$ type recurrence with $a_n=0$, $\forall ~ n$ and given by \eqref{special R1}. In the view of recurrence \eqref{special R1}, \Cref{Theorem s_k_x R1 pp} implies the following corollary:
\begin{corollary}\label{corollary s_k_x R1 pp}
The following formulas hold:
\begin{align*}
	{\mathcal{P}}_{n+1}(x;\mu_k,\nu_k) &= \mathcal{P}_{n+1}(x) - \mathcal{S}_k(x)\mathcal{P}^{(k+1)}_{n-k}(x), \qquad \mbox{$n \geq k$}, \\
\text{and} \quad 
{\mathcal{Q}}_{n+1}(x;\mu_k,\nu_k) &= \mathcal{Q}_{n+1}(x) + \hat{\mathcal{S}}_k(x)\mathcal{Q}^{(k+1)}_{n-k}(x), \qquad \mbox{$n \geq k$},
\end{align*}
where $\mathcal{S}_k(x)=\mu_k \mathcal{P}_k(x)+(\nu_k-1)\lambda_k x \mathcal{P}_{k-1}(x)$ and $\hat{\mathcal{S}}_k(x)=-\mu_k \mathcal{Q}_k(x)-(\nu_k-1)\lambda_k x \mathcal{Q}_{k-1}(x)$.
\end{corollary}
From \eqref{D_Pn}, we get $\mathcal{P}_{n+1}(x)$ and $\mathcal{P}^{(k+1)}_{n-k}(x)$ are linearly independent.
\begin{remark}
If $\mu_k=0$, then the degree of $\mathcal{S}_k(x)$ is $k$. Similarly, if $\nu_k=1$, then the degree of $\mathcal{S}_k(x)$ is $k$. This leads to the conclusion that the degree of $\mathcal{S}_k(x)$ is equal for both co-recursive and co-dilated cases. Note that this is not true for OPRL satisfying TTRR, see \cite[Theorem 2.1]{Castillo co-polynomials on real line 2015} {\rm(also see \cite{Paco perturbed recurrence 1990})}.
\end{remark}

\begin{theorem}\label{P_mu_nu to P_n common zeros R1 pp}
If $\mathcal{P}_{n+1}(x;\mu_k,\nu_k)$ and $\mathcal{P}_{n+1}(x)$ share a zero, then that zero will also be shared by $\mathcal{S}_k(x)$. 	
%$\mathcal{P}_{n+1}(x;\mu_k,\nu_k)$ and $\mathcal{P}_{n+1}(x)$ share at most the zeros of $\mathcal{S}_k$.
%and $\mathcal{P}_{k-1}$
\end{theorem}
\begin{proof}
Assume $\mathcal{P}_{n+1}(x;\mu_k,\nu_k)$ and $\mathcal{P}_{n+1}(x)$ share a common zero $\beta$ that is not the zero of $\mathcal{S}_k$. Let Y represent the set of zeros of $\mathcal{S}_k(x)$. Then, because $\beta \in \mathbb{C}\backslash (X \cup Y)$, \Cref{Theorem s_k_x R1 pp} implies $\mathcal{P}^{(k+1)}_{n-k}(\beta)=0$, a violation of the linear independence of $\mathcal{P}_{n+1}(x)$ and $\mathcal{P}^{(k+1)}_{n-k}(x)$.
\end{proof}

Once it is established that $\mathcal{P}_{n}(x;\mu_k,\nu_k)$ and $\mathcal{P}_{n}(x)$ share some common zeros, the next goal is to find the location in the perturbed polynomial sequence that contains all of the common zeros.

\begin{theorem}\label{P_2n+1 P_n common zeros R1 pp}
If $c_n$ and $\lambda_{n}$ are constants in \eqref{special R1}, then the set of zeros of  $\mathcal{P}_{n}(x)$ and $\mathcal{Q}_{n}(x) $ is the subset of the set of zeros of $ \mathcal{P}_{2n+1}(x)$ and $ \mathcal{Q}_{2n}(x)$, respectively.
\end{theorem}
\begin{proof}
For $c_n=c$ and $\lambda_n=\lambda$, we have $\mathbf{T}_{i}= \begin{bmatrix}
	x-c & -\lambda x \\
	1 & 0
\end{bmatrix}, \forall ~i$. Hence, we can write 
\begin{align*}
	\mathbf{T}_{n}\ldots \mathbf{T}_{0}=\mathbf{T}_{0}\ldots \mathbf{T}_{n}
	=\begin{bmatrix}
		\mathcal{P}_{n+1}(x) & -\lambda x\mathcal{P}_{n}(x) \\
		\mathcal{Q}_{n+1}(x) & -\lambda x\mathcal{Q}_{n}(x)
	\end{bmatrix}, 
\end{align*}
where $\mathbf{T}_{n}$ (or $\mathbf{T}_{0}$) is given by
\begin{align*}
	\mathbf{T}_{n}=\begin{bmatrix}
		x-c & -\lambda x \\
		1 & 0
	\end{bmatrix}	=\begin{bmatrix}
		\mathcal{P}_{1}(x) & -\lambda x\mathcal{P}_{0}(x) \\
		1 & 0
	\end{bmatrix}  .
\end{align*}  Now, from \eqref{F_n+1 to T_0 R1 pp}, we obtain
\begin{align}
\begin{bmatrix}
	\mathcal{P}_{2n+1}(x) & -\mathcal{Q}_{2n+1}(x)\\
	{\mathcal{P}}_{2n}(x) & -{\mathcal{Q}}_{2n}(x)
\end{bmatrix}&= (\mathbf{T}_n \ldots \mathbf{T}_{0})
\begin{bmatrix}
	\mathcal{P}_{n}(x) & -\mathcal{Q}_{n}(x) \\
	{\mathcal{P}}_{n-1}(x) & -\mathcal{Q}_{n-1}(x)
\end{bmatrix}. \label{Theorem 4.3 eq 1}
\end{align} 
Therefore, from \eqref{Theorem 4.3 eq 1}, we get
\begin{align*}
\begin{bmatrix}
	\mathcal{P}_{2n+1}(x) & -\mathcal{Q}_{2n+1}(x)\\
	{\mathcal{P}}_{2n}(x) & -{\mathcal{Q}}_{2n}(x)
\end{bmatrix} &=
\begin{bmatrix}
	\mathcal{P}_{n+1}(x) & -\lambda x\mathcal{P}_{n}(x) \\
	\mathcal{Q}_{n+1}(x) & -\lambda x\mathcal{Q}_{n}(x)
\end{bmatrix} \begin{bmatrix}
	\mathcal{P}_{n}(x) & -\mathcal{Q}_{n}(x) \\
	{\mathcal{P}}_{n-1}(x) & -\mathcal{Q}_{n-1}(x)
\end{bmatrix}. \\
\Longrightarrow \mathcal{P}_{2n+1}(x) &= \mathcal{P}_{n}(x)[\mathcal{P}_{n+1}(x)-\lambda x \mathcal{P}_{n-1}(x)],\\
\mathcal{Q}_{2n}(x) &= \mathcal{Q}_{n}(x)[\mathcal{Q}_{n+1}(x)-\lambda x\mathcal{Q}_{n-1}(x)],
\end{align*}
which means $\mathcal{P}_{n}(x)$ is a factor of $\mathcal{P}_{2n+1}(x)$ and $\mathcal{Q}_{n}(x)$ is a factor of $\mathcal{Q}_{2n}(x)$ and the conclusion follows.
%For $c_n=c$ and $\lambda_n=\lambda$, from \eqref{P_n+1 = A_n P_n}, we have
%\begin{align*}
%	\begin{bmatrix}
%		\mathcal{P}_{2n+1}(x) \\
%		{\mathcal{P}}_{2n}(x)
%	\end{bmatrix}&=
%	\begin{bmatrix}
%		x-c & -\lambda x \\
%		1 & 0
%	\end{bmatrix} \begin{bmatrix}
%		\mathcal{P}_{2n}(x) \\
%		{\mathcal{P}}_{2n-1}(x)
%	\end{bmatrix} \\
%	&=\begin{bmatrix}
%		\mathcal{P}_{1}(x) & -\lambda x\mathcal{P}_{0}(x) \\
%		1 & 0
%	\end{bmatrix} \begin{bmatrix}
%		x-c & -\lambda x \\
%		1 & 0
%	\end{bmatrix} \begin{bmatrix}
%		\mathcal{P}_{2n-1}(x) \\
%		{\mathcal{P}}_{2n-2}(x)
%	\end{bmatrix} \\
%	&=\begin{bmatrix}
%		\mathcal{P}_{2}(x) & -\lambda x\mathcal{P}_{1}(x) \\
%		q(x) & s(x)
%	\end{bmatrix} \begin{bmatrix}
%		\mathcal{P}_{2n-1}(x) \\
%		{\mathcal{P}}_{2n-2}(x)
%	\end{bmatrix}.
%\end{align*}
%After $n-1$ interations as above, we get
%\begin{align*}
%	\begin{bmatrix}
%		\mathcal{P}_{2n+1}(x) \\
%		{\mathcal{P}}_{2n}(x)
%	\end{bmatrix}&=
%	\begin{bmatrix}
%		\mathcal{P}_{n+1}(x) & -\lambda x\mathcal{P}_{n}(x) \\
%		q'(x) & s'(x)
%	\end{bmatrix} \begin{bmatrix}
%		\mathcal{P}_{n}(x) \\
%		{\mathcal{P}}_{n-1}(x)
%	\end{bmatrix}, \\
%	\mathcal{P}_{2n+1}(x) &= \mathcal{P}_{n}(x)[\mathcal{P}_{n+1}(x)-\lambda x\mathcal{P}_{n-1}(x)],
%\end{align*}
%which means $\mathcal{P}_{n}(x)$ is a factor of $\mathcal{P}_{2n+1}(x)$ and the proof is complete.
\end{proof}

\begin{proposition}\label{corollary zeros sharing R2 pp}
The polynomials $\mathcal{P}_{n}(x;\mu_k)$ and $\mathcal{P}_{n}(x)$ share $k$ zeros, that are identical to the zeros of $\mathcal{P}_{k}(x)$. Furthermore, the set of zeros of  $\mathcal{P}_{k}(x)$ and $\mathcal{Q}_{k}(x) $ is the subset of the set of zeros of $ \mathcal{P}_{2k+1}(x)$ and $ \mathcal{Q}_{2k}(x)$, respectively, for $c_n=c$ and $\lambda_n=\lambda$.
\end{proposition}
\begin{proof}
The first part follows from \Cref{P_mu_nu to P_n common zeros R1 pp}. Proceeding as in the proof of \Cref{P_2n+1 P_n common zeros R1 pp}, we obtain
\newline
$ \displaystyle
\begin{bmatrix}
	\mathcal{P}_{2k+1}(x;\mu_k) & -\mathcal{Q}_{2k+1}(x;\mu_k)\\
	{\mathcal{P}}_{2k}(x;\mu_k) & -\mathcal{Q}_{2k}(x;\mu_k)
\end{bmatrix}
$
\begin{align*}
&=\mathbf{T}_{2k+1} \ldots \mathbf{T}_{k+1} \mathbf{T}_k(\mu_k,\nu_k=1)  \begin{bmatrix}
	\mathcal{P}_{k}(x) & -\mathcal{Q}_{k}(x)\\
	{\mathcal{P}}_{k-1}(x) & -{\mathcal{Q}}_{k-1}(x)
\end{bmatrix} \\
&=\begin{bmatrix}
	\mathcal{P}_{k}(x) & -\lambda x\mathcal{P}_{k-1}(x) \\
	\mathcal{Q}_{k}(x) & -\lambda x\mathcal{Q}_{k-1}(x)
\end{bmatrix}(\mathbf{T}_k+\mathbf{N}_k)
\begin{bmatrix}
	\mathcal{P}_{k}(x) & -\mathcal{Q}_{k}(x)\\
	{\mathcal{P}}_{k-1}(x) & -{\mathcal{Q}}_{k-1}(x)
\end{bmatrix}\\
&=
\begin{bmatrix}
	\mathcal{P}_{k+1}(x)-\mu_k\mathcal{P}_{k}(x) & -\lambda x\mathcal{P}_{k}(x) \\
	\mathcal{Q}_{k+1}(x)-\mu_k\mathcal{Q}_{k}(x) & -\lambda x\mathcal{Q}_{k}(x)
\end{bmatrix} \begin{bmatrix}
	\mathcal{P}_{k}(x) & -\mathcal{Q}_{k}(x)\\
	{\mathcal{P}}_{k-1}(x) & -{\mathcal{Q}}_{k-1}(x)
\end{bmatrix}.
\end{align*}
This implies
\begin{align*}
&\mathcal{P}_{2k+1}(x;\mu_k) = \mathcal{P}_{k}(x)[\mathcal{P}_{k+1}(x)- \mu_k \mathcal{P}_{k}(x) -\lambda x\mathcal{P}_{k-1}(x)] \quad \mbox{and} \\
&\mathcal{Q}_{2k}(x;\mu_k) = \mathcal{Q}_{k}(x)[\mathcal{Q}_{k+1}(x)- \mu_k \mathcal{Q}_{k}(x) -\lambda x\mathcal{Q}_{k-1}(x)]. \qedhere
\end{align*}
%\begin{align*}
%	\begin{bmatrix}
%		\mathcal{P}_{2k+1}(x;\mu_k) \\
%		{\mathcal{P}}_{2k}(x;\mu_k)
%	\end{bmatrix}&=
%	\begin{bmatrix}
%		\mathcal{P}_{k+1}(x)-\mu_k\mathcal{P}_{k}(x) & -\lambda x\mathcal{P}_{k}(x) \\
%		q''(x) & s''(x)
%	\end{bmatrix} \begin{bmatrix}
%		\mathcal{P}_{k}(x) \\
%		{\mathcal{P}}_{k-1}(x)
%	\end{bmatrix}, \\
%	\mathcal{P}_{2k+1}(x;\mu_k) &= \mathcal{P}_{k}(x)[\mathcal{P}_{k+1}(x)- \mu_k \mathcal{P}_{k}(x) -\lambda x\mathcal{P}_{k-1}(x)]. \qedhere
%\end{align*}
\end{proof}

\begin{proposition}\label{Containment prop co-dilated}
For some $n \geq k$, the $k-1$ zeros of $\mathcal{P}_{k-1}(x)$ are shared by the co-dilated polynomials $\mathcal{P}_{n}(x;\nu_k)$ and the polynomials $\mathcal{P}_{n}(x)$. Specifically, for $c_n=c$ and $\lambda_n=\lambda$, the set of zeros of $\mathcal{P}_{k-1}(x)$, $\mathcal{Q}_{k-1}(x)$, $\mathcal{P}_{k}(x)$ and $\mathcal{Q}_{k}(x)$ is the subset of the set of zeros of $\mathcal{P}_{2k-1}(x;\nu_k)$, $\mathcal{Q}_{2k-2}(x;\nu_k)$, $\mathcal{P}_{2k+1}(x;\nu_k)$ and $\mathcal{Q}_{2k}(x;\nu_k)$, respectively.	
\end{proposition}
\begin{proof}
The first part is a direct consequence of \Cref{P_mu_nu to P_n common zeros R1 pp}. Using the conclusion of \Cref{Theorem s_k_x R1 pp} and the technique presented in \Cref{P_2n+1 P_n common zeros R1 pp}, we get
\newline
$ \displaystyle\begin{bmatrix}
	\mathcal{P}_{2k-1}(x;\nu_k) & -\mathcal{Q}_{2k-1}(x;\nu_k)\\
	{\mathcal{P}}_{2k-2}(x;\nu_k) & -{\mathcal{Q}}_{2k-2}(x;\nu_k)
\end{bmatrix}
$
\begin{align*}
&= \mathbf{T}_{2k-1} \ldots \mathbf{T}_{k+1} \mathbf{T}_k(\mu_k=0,\nu_k)  \begin{bmatrix}
	\mathcal{P}_{k}(x) & -\mathcal{Q}_{k}(x)\\
	{\mathcal{P}}_{k-1}(x) & -{\mathcal{Q}}_{k-1}(x)
\end{bmatrix}\\
&=\begin{bmatrix}
	\mathcal{P}_{k-2}(x) & -\lambda x\mathcal{P}_{k-3}(x) \\
	\mathcal{Q}_{k-2}(x) & -\lambda x\mathcal{Q}_{k-3}(x)
\end{bmatrix}(\mathbf{T}_k+\mathbf{N}_k)
\begin{bmatrix}
	\mathcal{P}_{k}(x) & -\mathcal{Q}_{k}(x)\\
	{\mathcal{P}}_{k-1}(x) & -{\mathcal{Q}}_{k-1}(x)
\end{bmatrix}\\
&=
\begin{bmatrix}
	\mathcal{P}_{k-1}(x) & -\nu_k\lambda x\mathcal{P}_{k-2}(x) \\
	\mathcal{Q}_{k-1}(x) & -\nu_k\lambda x\mathcal{Q}_{k-2}(x)
\end{bmatrix} \begin{bmatrix}
	\mathcal{P}_{k}(x) & -\mathcal{Q}_{k}(x)\\
	{\mathcal{P}}_{k-1}(x) & -{\mathcal{Q}}_{k-1}(x)
\end{bmatrix}.
\end{align*}
This implies
\begin{align*}
&\mathcal{P}_{2k-1}(x;\nu_k) = \mathcal{P}_{k-1}(x) [\mathcal{P}_{k}(x)-\nu_k\lambda x\mathcal{P}_{k-2}(x)] \quad \mbox{and}\\
&\mathcal{Q}_{2k-2}(x;\nu_k) = \mathcal{Q}_{k-1}(x) [\mathcal{Q}_{k}(x)-\nu_k\lambda x\mathcal{Q}_{k-2}(x)].	
\end{align*}
The remaining part of the proposition follows along similar lines.
%\begin{align*}
%	\begin{bmatrix}
%		\mathcal{P}_{2k-1}(x;\nu_k) \\
%		{\mathcal{P}}_{2k-2}(x;\nu_k)
%	\end{bmatrix}&=
%	\begin{bmatrix}
%		\mathcal{P}_{k-1}(x) & -\nu_k\lambda x\mathcal{P}_{k-2}(x) \\
%		q''(x) & s''(x)
%	\end{bmatrix} \begin{bmatrix}
%		\mathcal{P}_{k}(x) \\
%		{\mathcal{P}}_{k-1}(x)
%	\end{bmatrix}, \\
%	\mathcal{P}_{2k-1}(x;\nu_k) &= \mathcal{P}_{
%		k-1}(x)[\mathcal{P}_{k}(x)-\nu_k\lambda x\mathcal{P}_{k-2}(x)].
%\end{align*}
%The last part of the proposition also follows from a similar analysis.
\end{proof}

\begin{remark}
It should be noted that, for co-recursive polynomials, the set of zeros of $\mathcal{P}_{k-1}(x)$ is not contained in the set of zeros of $\mathcal{P}_{2k-1}(x;\mu_k)$. However, this fact holds for co-dilated polynomials {\rm(see \Cref{Containment prop co-dilated})}.
%Similar to the relationship $Z[\mathcal{P}_{k-1}(x)] \subset Z[\mathcal{P}_{2k-1}(x;\nu_k)]$ for co-dilated polynomials given in \Cref{Containment prop co-dilated}, a relationship like $Z[\mathcal{P}_{k-1}(x)] \subset Z[\mathcal{P}_{2k-1}(x;\mu_k)]$ for co-recursive polynomials is not possible. 
This is due to the fact that a similar calculation for co-recursive polynomials produces
\begin{align*}
	\mathcal{P}_{2k-1}(x;\mu_k)&= \mathcal{P}_{k}(x)[\mathcal{P}_{k-1}(x)-\mu_k\mathcal{P}_{k-2}(x)] -\lambda x\mathcal{P}_{k-2}(x)\mathcal{P}_{k-1}(x).
\end{align*}
\end{remark}

We need the following result given in \cite{walter toda laurent 2002} to prove our next result.

\begin{theorem}\cite[Lemma 2.1]{walter toda laurent 2002}\label{walter zeros lemma}
The zeros $x^{(n)}_j$, $j=1,2, \ldots n$ of $\mathcal{P}_n(x)$ are real, simple and positive. Assuming the ordering $x^{(n)}_{j-1} < x^{(n)}_{j}$, we have the interlacing property
\begin{align*}
	0<	x^{(n+1)}_{1} < x^{(n)}_{1}<x^{(n+1)}_2<x^{(n)}_2 < \cdots < x^{(n+1)}_{n} < x^{(n)}_{n} <x^{(n+1)}_{n+1}, \quad  n \geq 1.
\end{align*}
\end{theorem}

\begin{theorem}\label{interlacing theorem mu}
Let $l$ be the number of non common zeros between $\mathcal{P}_{n}(x;\mu_k)$ and $\mathcal{P}_{n}(x)$, $n \geq k$. Suppose $x^{(n)}_j(\mu)$ and $x^{(n)}_j$, $j=1,2, \ldots l$, are the zeros corresponding to $\mathcal{P}_{n}(x;\mu_k)$ and $\mathcal{P}_{n}(x)$. If $\mu < 0$, then 
\begin{align}\label{interlacing non common}
	%x^{(n)}_l(\mu) < x^{(n)}_l < x^{(n)}_{l-1}(\mu) < x^{(n)}_{l-1}< \ldots < x^{(n)}_{1}(\mu) < x^{(n)}_{1}
	x^{(n)}_{1}(\mu) <x^{(n)}_{1}<x^{(n)}_{2}(\mu)< \cdots < x^{(n)}_{l-1} < x^{(n)}_l(\mu) < x^{(n)}_{l},
\end{align}
where the role of the zeros $x^{(n)}_j(\mu)$ and $x^{(n)}_j$, $j=1,2, \ldots l$ gets intercharged when $\mu > 0$.
\end{theorem}
\begin{proof}
Since $n \geq k$, using \eqref{P_n+1 mu_k nu_k R1 pp}, \eqref{T_k + N_k} and \eqref{casarotti determinant}, we can write
\begin{align}
&\begin{array}{|cc|}
	\mathcal{P}_{n+1}(x;\mu_k) & \mathcal{P}_{n+1}(x) \\
	\mathcal{P}_{n}(x;\mu_k) & \mathcal{P}_{n}(x)
\end{array} = |\mathbf{T}_n \ldots \mathbf{T}_{k+1}| \begin{vmatrix}
	\mathcal{P}_{k+1}(x)-\mu_k\mathcal{P}_{k}(x) & \mathcal{P}_{k+1}(x) \\
	\mathcal{P}_{k}(x) & \mathcal{P}_{k}(x)
\end{vmatrix}. \nonumber\\
%&= \lambda_n(x^2+\omega^2)D(\mathcal{P}_{n-1}(x), \mathcal{P}_{n-1}(x;\mu_k)).\\
&\Longrightarrow \quad \mathcal{P}_{n}(x)\mathcal{P}_{n+1}(x;\mu_k)-\mathcal{P}_{n}(x;\mu_k)\mathcal{P}_{n+1}(x) = -\mu_k(\lambda_n \ldots \lambda_k)x^{n-k}\mathcal{P}^2_{k}(x).\label{D(P_n P_mu)}
\end{align}
We have $(-1)^j \mathcal{P}_{n+1}(x^{(n)}_j)> 0$, $j=1,2, \ldots n$. Recall that $x^{(n)}_j$, $j=1,2, \ldots n$ are the $n$ real zeros corresponding to $\mathcal{P}_{n}(x)$. First, consider the case when $\mathcal{P}_{n}(x;\mu_k)$ and $\mathcal{P}_{n}(x)$ have no common zeros. Now, when $\mu < 0$ and $\{\lambda_{n}\}_{n \geq 1}$ is a positive sequence of real numbers, From  \Cref{walter zeros lemma} and \eqref{D(P_n P_mu)}, we have $-\mathcal{P}_{n}(x^{(n)}_j;\mu_k)\mathcal{P}_{n+1}(x^{(n)}_j) > 0$, which implies 
\begin{align*}
	(-1)^{j+1}\mathcal{P}_{n}(x^{(n)}_j;\mu_k) > 0, \quad j \geq 1.
\end{align*}
Since the number of real zeros of $\mathcal{P}_{n}(x;\mu_k)$ cannot excced $n$ and $\mathcal{P}_{n}(x^{(n)}_1;\mu_k) > 0$, \eqref{interlacing non common} follows. For $\mu > 0$, a similar analysis leads to change of sign in \eqref{D(P_n P_mu)} and subsequently, the result follows.\par
The case when $\mathcal{P}_{n}(x;\mu_k)$ and $\mathcal{P}_{n}(x)$ have common zeros is dealt with in a similar manner by excluding such common zeros.
\end{proof}

\begin{remark}
In \cite{Castillo monotonicity R1 2015}, a similar result was proved by transforming the $R_I$ type recurrence relation into a three-term Frobenius type recursion. However, our approach differs because we have not used the Delsarte-Genin mapping discussed in the aforementioned paper. Instead, we have used the positivity of zeros in L-orthogonal polynomials, as demonstrated in \cite{walter toda laurent 2002}.
\end{remark}

\begin{theorem}
Let  $x^{(n)}_j$ and $x^{(n)}_j(\vec{\mu})$, $j=1,2,\ldots,n$ be the zeros of $\mathcal{P}_{n}(x)$ and $\mathcal{P}_{n}(x;\vec{\mu})$, respectively, where $\vec{\mu}=(\mu_0,\ldots, \mu_{n-1})$. If $\mu_k > 0$ {\rm(}or $<0${\rm)}, $k=0,1,\ldots,n-1$, then the zeros of $\mathcal{P}_{n}(x;\vec{\mu})$ are strictly increasing {\rm(}or strictly decreasing{\rm)} functions of $\mu_0,\ldots,\mu_{n-1}$.
\end{theorem}
\begin{proof}
Let $x^{(n)}_j(\mu_0)$, $j=1,2,\ldots,n$ be the zeros of $\mathcal{P}_{n}(x;\mu_0)$. Then, applying \Cref{interlacing theorem mu} and assuming $\mu_0 > 0$, we get
\begin{align*}
	x^{(n)}_j < x^{(n)}_j(\mu_0), \quad j=1,2,\ldots,n.
\end{align*}
Now, considering $\mathcal{P}_{n}(x;\mu_0)$ as a new sequence of $\mathcal{R}_{I}$ polynomials and perturbing $c_1$ by $\mu_1 > 0$, we have the polynomials $\mathcal{P}_{n}(x;\mu_0,\mu_{1})$ whose zeros are $x^{(n)}_j(\mu_0,\mu_{1})$, $j=1,2,\ldots,n$. Again, applying \Cref{interlacing theorem mu}, we obtain
\begin{align*}
	x^{(n)}_j(\mu_0) < x^{(n)}_j(\mu_0, \mu_1), \quad j=1,2,\ldots,n.
\end{align*}
Repeating same process $n-2$ times, we have
\begin{align*}
	x^{(n)}_j(\mu_0) < x^{(n)}_j(\mu_0, \mu_1) < \cdots < x^{(n)}_j(\mu_0,\ldots, \mu_{n-1}) \quad j=1,2,\ldots,n.
\end{align*}
To show that the zeros of $\mathcal{P}_{n}(x;\vec{\mu})$ are increasing fucntions of $\mu_k$, $k=0,1,\ldots,n-1$, we will replace $\mu_k$ by $\mu_k+\epsilon_k$, $\epsilon_k > 0$, $k=0,1,\ldots,n-1$ and establish that $x^{(n)}_j(\vec{\mu}) < x^{(n)}_j(\vec{\mu},\vec{\epsilon})$, $j=1,2,\ldots,n$, where $\vec{\epsilon}=(\epsilon_0,\ldots,\epsilon_{n-1})$ and $x^{(n)}_j(\vec{\mu},\vec{\epsilon})$ be the zeros of polynomial $\mathcal{P}_{n}(x;\vec{\mu}, \vec{\epsilon})$.
First of all, replacing $\mu_0$ by $\mu_0+\epsilon_0$, the polynomials $\mathcal{P}_{n}(x;\vec{\mu}, \epsilon_0)$ are obtained. Proceeding as per the proof of \Cref{interlacing theorem mu}, it can be shown that
\begin{align*}
	x^{(n)}_j(\vec{\mu}) < x^{(n)}_j(\vec{\mu}, \epsilon_0), \quad j=1,2,\ldots,n.
\end{align*}
where $x^{(n)}_j(\vec{\mu}, \epsilon_0)$, $j=1,2,\ldots,n$ are the zeros of $\mathcal{P}_{n}(x;\vec{\mu}, \epsilon_0)$. Further, let $x^{(n)}_j(\vec{\mu}, \epsilon_0, \epsilon_1)$, $j=1,2,\ldots,n$ be the zeros of polynomial $\mathcal{P}_{n}(x;\vec{\mu}, \epsilon_0, \epsilon_1)$ obtained on replacing $\mu_1$ by $\mu_1+\epsilon_1$. From \Cref{interlacing theorem mu}, we get
\begin{align*}
	x^{(n)}_j(\vec{\mu}, \epsilon_0) < x^{(n)}_j(\vec{\mu}, \epsilon_0, \epsilon_1), \quad j=1,2,\ldots,n.
\end{align*}
Using similar argument, we conclude that
\begin{align}
x^{(n)}_j(\vec{\mu}) < x^{(n)}_j(\vec{\mu}, \epsilon_0) < x^{(n)}_j(\vec{\mu}, \epsilon_0, \epsilon_1), < \cdots< x^{(n)}_j(\vec{\mu},\vec{\epsilon}) \quad j=1,2,\ldots,n,
\end{align} 
and the proof is complete.
\end{proof}

\begin{corollary}\label{Corolary monotonicity R1 pp}
Let $x^{(n)}_j(\mu_k,\mu_{k+1})$, $j=1,2,\ldots,n$ be the zeros of $\mathcal{P}_{n}(x;\mu_k,\mu_{k+1})$ with $\mu_k>0$ and $\mu_{k+1}>0$ obtained by perturbing two consecutive recurrence coefficients. Then for fixed $j$ and $n \geq k$, $x^{(n)}_j(\mu_k,\mu_{k+1})$ are strictly increasing functions of $\mu_k$ and $\mu_{k+1}$. Similarly, the zeros of $\mathcal{P}_{n}(x;\mu_k,\mu_{k+1})$ are strictly decreasing functions of $\mu_k$ and $\mu_{k+1}$ whenever $\mu_k<0$ and $\mu_{k+1}<0$.
\end{corollary}

\subsection{Perturbation in the chain sequence and Carath\'eodary function}\label{Concluding remarks}
Recurrence relation of the form
\begin{align}\label{r_n R1 paper}
\mathcal{P}_{n+1}(z) = ((1+i\beta_{n+1})z+(1-i\beta_{n+1}))\mathcal{P}_{n}(z)- 4d_{n+1}z\mathcal{P}_{n-1}(z),
\end{align}
with initial conditions $\mathcal{P}_0(z) = 1$ and $\mathcal{P}_1(z) = (1+i\beta_{1})z+(1-i\beta_{1})$ are studied extensively in \cite{Castillo monotonicity R1 2015, Castillo costa ranga veronese Favard theorem 2014} to cite a few. The existence of a unique non-trivial probability measure supported on the unit disc is known since the classical work of Delsarte and Genin, and has recently been revisited using a different technique, i.e., using Markov's theorem, in \cite{Castillo Constr. Approx. 2022}.  Further, when $\{\beta_n\}_{n=1}^\infty$ is a sequence of real numbers, we have 
\begin{align}\label{phi_n from r_n R1 pp}
\varphi_0(z)=1, \quad \varphi_{n}(z)\prod_{k=1}^{n}(1+i\beta_{k})=\mathcal{P}_{n}(z)-2(1-m_n)\mathcal{P}_{n-1}(z),
\end{align}
as the corresponding sequence of monic OPUC where $\{m_n\}_{n=1}^\infty$ is the minimal parameter sequence of the the chain sequence $\{d_{n+1}\}_{n=1}^\infty$. Perturbation of co-recursive type in $\beta_n$ is investigated in \cite{Castillo monotonicity R1 2015} from the view point of interlacing properties and some new inequalities involving zeros are derived. Note that, under certain restrictions, \eqref{r_n R1 paper} can be viewed as a special case of \eqref{R1}. The polynomials $\mathcal{P}_{n+1}(z)$ generated by \eqref{r_n R1 paper} are called the para-orthogonal polynomials on the unit circle (POPUC), and the most general results related to zeros of POPUC can be found in \cite{Castillo LAA 2019, Castillo Petronilho PAMS 2018}. This section aims to study the co-dilation in \eqref{r_n R1 paper} and we will see how a perturbation in the chain sequence $\{d_{n+1}\}_{n \geq 1}$ leads to some interesting consequences in function theory and orthogonal polynomials. \par 

The Hermitian Perron-Carath\'eodary fractions, or HPC-fractions, arising in the study of the connection between the Szeg\H{o} polynomials and the continued fractions are given as
\begin{align}\label{HPC fraction}
\delta_0 - \frac{2\delta_{0}}{1} \mathbin{\genfrac{}{}{0pt}{}{}{+}} \frac{1}{\bar{\delta}_1 z} \mathbin{\genfrac{}{}{0pt}{}{}{+}} \frac{(1-|\delta_1|^2)z}{\delta_1} \mathbin{\genfrac{}{}{0pt}{}{}{+}} \frac{1}{\bar{\delta}_2 z} \mathbin{\genfrac{}{}{0pt}{}{}{+}} \frac{(1-|\delta_2|^2)z}{\delta_2}\ldots,
\end{align}
and are completely determined by $\delta_{n} \in  \mathbb{C}$, where $\delta_{0} \neq 0 $ and $\delta_{n} \neq 1 $, $n \geq 1$. When $\delta_{0} > 0 $ and $|\delta_{n}| < 1 $, $n \geq 1$, \eqref{HPC fraction} is known as the positive PC fraction (PPC-fration). Then $\phi_n(z)$ are the odd ordered denominators $\mathcal{Q}_{2n+1}(z)$ and $\phi^*_n(z)$ are the even ordered denominators $\mathcal{Q}_{2n}(z)$, where $\mathcal{Q}_n(z)$, a polynomial of degree $n$, is the denominator of $n^{th}$ approximant of PPC-fration \cite{Jones njstad thron 1989}. Having set $\delta_{n}=\phi_n(0)$ as reflection coefficients, the following are the recurrence relations for Szeg\H{o} polynomials
\begin{align*}
	\phi^*_n(z) &= \bar{\delta}_n z\phi_{n-1}(z)+\phi^*_{n-1}(z), \\
	\phi_n(z) &= \delta_n\phi^*_{n}(z)+(1-|\delta_n|^2)z\phi_{n-1}(z), \quad n \geq 1.
\end{align*}
Under the condition 
\begin{align}\label{delta_n condition}
\delta_{n+1}\delta_{n}=\delta_{n+1}-\delta_{n}, \quad \delta_{0} > 0  ~ \text{and} ~ |\delta_{n}| < 1 , \quad n \geq 1,
\end{align}
the chain sequence $\{d_{n+1}\}_{n \geq 1}$ is such that $d_{n+1}=(1-m_n)m_{n+1}$, where $m_{n}=(1-\delta_n)/2$, $n \geq 1$, is the minimal parameter sequence, i.e.,
\begin{align}\label{lambda 1 and 2}
	d_{n+1}=\dfrac{1}{4}(1+\delta_{n})(1-\delta_{n+1})=\dfrac{1}{4}(1-2\delta_{n}\delta_{n+1}).
	% \lambda_{n+1}^{(1)}=\dfrac{1}{4}(1-\alpha_{n-2})(1+\alpha_{n-1}), \qquad
\end{align}
A sequence of para-orthogonal polynomials satisfying 
\begin{align}\label{R1 kiran}
\mathcal{P}_{n+1}(z) &= (z+1)\mathcal{P}_n(z)-4d_{n+1}z\mathcal{P}_{n-1}(z), \quad n \geq 1,
%\mathcal{P}_{n+1}^{(2)}(z) &= (z+1)\mathcal{P}_n^{(2)}(z)-4\lambda_{n+1}^{(2)}z\mathcal{P}_{n-1}^{(2)}(z), \quad n \geq 1, 
\end{align}
with initial conditions $\mathcal{P}_0(z)=1$ and $\mathcal{P}_1(z)=z+1$ has been obtained in \cite[Proposition 3.1]{KKB ranga swami 2016} together with a sequence of Szeg\H{o} polynomials $\phi_n(z)$ having $\delta_{n} \in \mathbb{R}$ as the reflection coefficients and satisfying \eqref{delta_n condition}. The Carath\'eodary function 
\begin{align*}
\mathcal{C}(z)= \dfrac{1+(1-2\sigma)z}{1-z}, \quad |z|< 1, \quad 0 < \sigma <1,
\end{align*}
corresponds to a PPC-fraction with the parameter $\delta_n = -\frac{1}{n+\frac{\sigma}{1-\sigma}}$ obeying \eqref{delta_n condition}. \par 
%Here, $\alpha_{-1}=1$. 
Now, from \eqref{lambda 1 and 2}, we have
\begin{align*}
	\delta_{n+1}= \dfrac{1-4d_{n+1}}{2\delta_{n}}.
\end{align*}
%or equivalently,
%\begin{align}%\label{alpha n continued fraction}
%\begin{split}
%\alpha_{n-1} &= -1 + \frac{4\lambda_{n+1}^{(1)}}{2} \mathbin{\genfrac{}{}{0pt}{}{}{-}} \frac{4\lambda_{n}^{(1)}}{2} \mathbin{\genfrac{}{}{0pt}{}{}{+}}   \mathbin{\genfrac{}{}{0pt}{}{}{\cdots}} \frac{2\lambda_2^{(1)}}{1}, \\
%\alpha_{n} &= 1 - \frac{4\lambda_{n+1}^{(2)}}{2} \mathbin{\genfrac{}{}{0pt}{}{}{-}} \frac{4\lambda_{n}^{(2)}}{2} \mathbin{\genfrac{}{}{0pt}{}{}{-}}  \mathbin{\genfrac{}{}{0pt}{}{}{\cdots}}  \frac{2\lambda_2^{(2)}}{1}. 
%\end{split}
%\end{align}
Therefore, from the sequence $\{d_{n+1}\}_{n \geq 1}$, the sequence $\{\delta_{n} \}_{n \geq 0}$ can be completely determined. Perturbing \eqref{R1 kiran} as \eqref{co-dilated condition R1} such that the new sequence $\{d'_{n+1}\}_{n \geq 1}$ is again a positive chain sequence, we get a new sequence of the Szeg\H{o} polynomials, say $\hat{\phi}_n(z)$ with $\{\hat{\delta}_{n}\}_{n \geq 0}$ as the reflection coefficients. Suppose that the perturbation is for $n=k$, then $d'_{k+1}= \nu_{k+1}d_{k+1}$ and
\begin{align*}
\hat{\delta}_{k+1} =\dfrac{1-4d'_{k+1}}{2\hat{\delta}_{k}} = \dfrac{1-4\nu_{k+1}d_{k+1}}{2\hat{\delta}_{k}}	~ \text{and} ~ \hat{\delta}_{n+1} =\dfrac{1-4d_{n+1}}{2\hat{\delta}_{n}}, \quad n > k.
\end{align*}
With these $\hat{\delta}_{n}$'s as the parameters of a PPC-fraction, there corresponds a Carath\'eodary function, say $\mathcal{\hat{C}}(z)$. Observe that the structure of $\mathcal{\hat{C}}(z)$ depends on the choice of $\nu_{k+1}$ and the level at which perturbation is done. This makes it difficult to determine exact form of $\mathcal{\hat{C}}(z)$. However, for some particular cases, we can find its precise form. One such choice is proposed in the next result by summarizing the above details.
%\begin{align}\label{alpha n hat}
%\begin{split}
%\hat{\alpha}_{n-1} &= -1 + \frac{4\lambda_{n+1}^{(1)}}{2} \mathbin{\genfrac{}{}{0pt}{}{}{-}} \frac{4\lambda_{n}^{(1)}}{2} \mathbin{\genfrac{}{}{0pt}{}{}{+}} \mathbin{\genfrac{}{}{0pt}{}{}{\cdots}} \frac{4\nu_{k+1}^{(1)}\lambda_{k+1}^{(1)}}{2} \mathbin{\genfrac{}{}{0pt}{}{}{+}}  \mathbin{\genfrac{}{}{0pt}{}{}{\cdots}} \frac{2\lambda_2^{(1)}}{1}, \\
%\hat{\alpha}_{n} &= 1 - \frac{4\lambda_{n+1}^{(2)}}{2} \mathbin{\genfrac{}{}{0pt}{}{}{-}} \frac{4\lambda_{n}^{(2)}}{2} \mathbin{\genfrac{}{}{0pt}{}{}{-}} \mathbin{\genfrac{}{}{0pt}{}{}{\cdots}} \frac{4\nu_{k+1}^{(2)}\lambda_{k+1}^{(2)}}{2} \mathbin{\genfrac{}{}{0pt}{}{}{+}} \mathbin{\genfrac{}{}{0pt}{}{}{\cdots}} \ \frac{2\lambda_2^{(2)}}{1},
%\end{split}
%\end{align}

\begin{proposition}
%The Szeg\H{o} polynomials $\hat{\phi}_n(z)$ with Verblunsky coefficients $\{\hat{\delta}_{n}\}$ and hence 
Consider the sequence $\{\delta_{n} \}_{n \geq 0}$ of reals under the restriction $\delta_{n+1}\delta_{n}=\delta_{n+1}-\delta_{n}$, $\delta_{0} > 0$ and $|\delta_{n}| < 1$, $n \geq 1$. If $\mathcal{C}(z)$ is the Carath\'eodory function whose PPC-fraction can be obtained from $\delta_n$, then $\hat{\delta}_{n}$ given as
%\begin{align*}
%\hat{\alpha}_{n-1} &= \alpha_{n-1}, \quad n < k, \\
%\hat{\alpha}_{k-1} &= \alpha_{k-1}+\dfrac{4(\nu_{k+1}^{(1)}-1)\lambda_{k+1}^{(1)}}{1-\alpha_{k-2}}, \qquad 
%\hat{\alpha}_{n-1} =-1+ \dfrac{4\lambda_{n+1}^{(1)}}{1-\hat{\alpha}_{n-2}}, \quad n > k,
%\end{align*}
%and
\begin{align*}
\hat{\delta}_{n+1} &= \delta_{n+1}, \hspace{7.9cm} n < k, \\
\hat{\delta}_{k+1} &= \delta_{k+1}+\dfrac{2(1-\nu_{k+1})d_{k+1}}{\delta_{k}}, \qquad
\hat{\delta}_{n+1} =\dfrac{1-4d_{n+1}}{2\hat{\delta}_{n}}, \quad n > k,
\end{align*}
gives the PPC-fraction corresponding to Carath\'eodary function $\mathcal{\hat{C}}(z)$. Further, if $\nu_{k+1}=\dfrac{1+2\delta_{k}\delta_{k+1}}{1-2\delta_{k}\delta_{k+1}}$, then $\hat{\delta}_{k+1}=-\delta_{k+1}$, $n \geq k$. In particular, if $\nu_{2}=\dfrac{1+2\delta_{1}\delta_{2}}{1-2\delta_{1}\delta_{2}}$, then $\hat{\delta}_{n}=-\delta_{n}=\dfrac{1}{n+\frac{\sigma}{1-\sigma}}$, $n \geq 1$ and the corresponding Carath\'eodary function is given by $1/\mathcal{C}(z)$.
\end{proposition}
The OPUC with $\{-\delta_{n} \}_{n \geq 0}$ as reflection coefficients and corresponding Carath\'eodary function has been studied in \cite{Ronning 1992}.
%\begin{proof}
%Direct computation using \eqref{DGPOP}, \eqref{alpha n continued fraction} and \eqref{alpha n hat} gives the result. Details of a similar kind of proof can be found in \cite{Castillo perturbed polynomials via szego transform 2017}.
%\end{proof}
%DG1POP and DG2POP polynomials find their application in frequency analysis problems \cite{Bracciali Li Ranga 2005}. Their perturbed version will also be useful in study of similar physical problems.

\section{Illustrative examples}\label{Illustrative examples}
\begin{example}\label{example 1 R1 paper}
Consider the sequence of $R_I$ polynomials generated by the recurrence
\begin{align}\label{Example 2 recurrence R1 paper}
	\mathcal{P}_{n+1}(z) = \left(z-\dfrac{b-c-n}{b+n}\right)\mathcal{P}_{n}(z) -\dfrac{n(n+c-1)}{(b+n-1)(b+n)}z\mathcal{P}_{n-1}(z),
\end{align}
with initial conditions $\mathcal{P}_{0}(z) = 1$ and $\mathcal{P}_1(z) = z-\dfrac{b-c}{b}$. 

%Hence,
%\begin{align*}
%	c_n = \dfrac{b-c-n}{b+n}, \quad \lambda_{n} = \dfrac{n(n+c-1)}{(b+n-1)(b+n)}, \quad a_n = 0.
%\end{align*}
\end{example}
If $(a)_n$ is the pochhammer symbol and $F$ represents the Gaussian hypergeometric function, use of the contiguous relation
\begin{align*}
	(c-a)F(a-1,b;c;z)=(c-2a-(b-a)z)F(a,b;c;z)+a(1-z)F(a+1,b;c;z),
\end{align*}
shows that the monic polynomial
\begin{align*}
	\mathcal{P}_{n}(z) = \dfrac{(c)_n}{(b)_n}F(-n,b;c;1-z),
\end{align*}
satisfies \eqref{Example 2 recurrence R1 paper} \cite{Ranga PAMS 2010}. For $b= \eta+1$ and $c= 2\eta+2$ with $\eta > -1/2$, \eqref{Example 2 recurrence R1 paper} reduces to 
\begin{align*}
\mathcal{P}_{n+1}(z) = (z+1)\mathcal{P}_{n}(z)-\dfrac{n(2\eta+n+1)}{(\eta+n)(\eta+n+1)}z \mathcal{P}_{n-1}(z),
\end{align*}
which is of the form \eqref{r_n R1 paper} with $\beta_{n}=0$.
\subsection{Co-recursion} It can be easily verified that, for $\eta=0$, the solution of the recurrence
\begin{align}\label{r_n corecursion}
&\mathcal{P}_{n+1}(z) = (z+1)\mathcal{P}_{n}(z)-z \mathcal{P}_{n-1}(z),
\end{align}
is given by
\begin{align}
\mathcal{P}_{n}(z)=\dfrac{z^{n+1}-1}{z-1}, \quad n \geq 0.
\end{align}

Now, let us make a perturbation at the beginning of the sequence. If $\mu > 0$, we write $\mathcal{P}_1(z) = z+1-\mu$. Consequently, we have
\begin{align*}
	\mathcal{P}_{n+1}(z;\mu) = \mathcal{P}_{n+1}(z)-\mu \mathcal{P}_{n}(z).
\end{align*}
Note that this is the case for perturbation for $k=0$. Therefore, $\mathcal{P}_{0}(z)=1$ and $\mathcal{P}^{(0)}_{n}(z)=\mathcal{P}_{n}$ produce the above relation. If $\mu = 1$, then
\begin{align*}
	\mathcal{P}_{n+1}(z;1) &= \mathcal{P}_{n+1}(z)-\mathcal{P}_{n}(z) = z^{n+1} .\nonumber
\end{align*} 
The polynomial $\mathcal{P}_{n+1}(z;1)$ satisfies \eqref{r_n corecursion} but with perturbed initial conditions $\mathcal{P}_{0}(z;1) = 1$, $\mathcal{P}_1(z;1) = z$, which shows the validity of our results in \Cref{Structural relations}. 

\subsection{Co-dilation and chain sequences}
In this part, we will see the effect of the co-dilation in the chain sequence on the corresponding parameter sequences, Szeg\H{o} polynomials, reflection coefficients, and associated measures.
%Closely related to the chain sequences is the important concept of complementary chain sequences which is defined as the following:
%\begin{definition}\cite{KKB ranga swami 2016}
%If $\{m_{n} \}_{n \geq 0}$ is the minimal parameter sequence of $\{d_{n} \}_{n \geq 1}$, then the sequence $\{a_{n} \}_{n \geq 1}$, whose minimal parameter sequence is a new sequence $\{k_{n} \}_{n \geq 0}$, is the complementary chain sequence of $\{d_{n} \}_{n \geq 1}$ when $k_0 = 0$ and $k_n = 1-m_n$, $n \geq 1$.
%\end{definition}
For $\eta > -1/2$, the positive chain sequence $\{d_{n+1} \}_{n \geq 1}$, where
\begin{align*}
	d_{n+1}= \dfrac{1}{4} \dfrac{n(2\eta+n+1)}{(\eta+n)(\eta+n+1)},
\end{align*}
has minimal parameter sequence $\{m_{n} \}_{n \geq 0}$, where $m_n=\dfrac{n}{2(\eta+n+1)}$ and a maximal parameter sequence $\{M_{n} \}_{n \geq 0}$, where $M_n=\dfrac{2\eta+n}{2(\eta+n)}$, which makes $\{d_{n+1} \}_{n \geq 1}$ non-SPPCS \cite{Bracciali ranga swami paraorthogonal 2016}. Clearly, for $\eta=0$, $\{d_{n+1}=1/4 \}_{n \geq 1}$ is not an SPPCS and its minimal and maximal parameter sequences are given as
\begin{align*}
m_{n}=\dfrac{n}{2n+2}, \quad n \geq 0, \quad \rm{and} \quad M'_{n}=\dfrac{1}{2}, \quad n \geq 0.
\end{align*} 
The corresponding para-orthogonal polynomials are
\begin{align*}
	\mathcal{P}_n(z)=\dfrac{z^{n+1}-1}{z-1}, \quad n \geq 0,
\end{align*}
and hence, the monic OPUC is given by
\begin{align*}
	\phi_{n}(z)=\dfrac{(n+1)z^{n}+nz^{n-1}+(n-1)z^{n-2}+\ldots+2z+1}{n+1}, \quad n \geq 0,
\end{align*}
with $\gamma_{n-1}=-\dfrac{1}{n+1}$ as reflection coefficients and $d\mu(z) = \dfrac{|z-1|^2}{4i\pi z}dz$ as the measure of orthogonality  \cite{Esmail Ranga 2018}. \par 
Consider the co-dilated chain sequence $\{1/2, 1/4, 1/4, \ldots \}$ that results from multiplication of $2$ in the first term of the chain sequence $\{1/4, 1/4, 1/4, \ldots \}$. The co-dilated chain sequence is an SPPCS as it is known to have a minimal parameter sequence $\{k'_{n} \}_{n \geq 0}$, where $k'_{0}=0 $ and $k'_{n}=1/2 $, $n \geq 1$, which is also the maximal parameter sequence.
So, from \eqref{r_n R1 paper}, the para-orthogonal polynomials in this case are obtained as
\begin{align*}
\tilde{\mathcal{P}}_n(z)&=z^n+1, \quad n \geq 1.
%q_n(z)& = \dfrac{z^n-1}{z-1},
\end{align*}
The reflection coefficients are zero and $\tilde{\phi}_{n}(z)= z^n$ are the Szeg\H{o}-Chebyshev polynomials with respect to standard Lebesgue measure. \par 
 %It is important to remark that this situation is same as discussed in \cite{KKB ranga swami 2016} when $\eta = 1$. In fact, the co-dilated chain sequence discussed above is the complementary chain sequence for the case $\eta = 1$ in \eqref{a_n+1 complementary}. 
% In this case, the Szeg\H{o} polynomials are 
%\begin{align*}
%	\phi'_{n}(z)= z^n+\left(\omega^{(1)}-\dfrac{n+2}{n+1}\right)z^{n-1}-\dfrac{\omega^{(1)}}{n+1}(z^{n-1}+ \ldots +z)-\dfrac{1}{n+1}, \quad n \geq 1,
%\end{align*}
%with the Verblunsky coefficients $-\phi'_{n}(0)=\dfrac{1}{n+1}$. 
%and the associated measure is $d\mu(\xi) = \dfrac{(1-\xi)(\xi-1)}{4i\pi \xi^2}d\xi$ \cite[Example 2]{Bracciali ranga swami paraorthogonal 2016}.

\subsection{Behaviour of zeros}
\begin{example}
	Let us consider the specialized $R_{I}$ type recurrence relation \eqref{special R1} with $c_n = \dfrac{c+n}{a-n-1}$, $n \geq 0$ and $\lambda_n =\dfrac{n(c+n-a)}{(a-n-1)(a-n)}$, $n \geq 1$.
\end{example}
With these values, we have 
\begin{align}\label{L-jacobi}
	\mathcal{P}^{(a,c)}_{n+1}(x) = \left(x-\dfrac{c+n}{a-n-1}\right)\mathcal{P}^{(a,c)}_n(x)- \dfrac{n(c+n-a)}{(a-n-1)(a-n)} x\mathcal{P}^{(a,c)}_{n-1}(x), \quad n\geq 1,
\end{align}
where $\mathcal{P}^{(a,c)}_{n}(x)$ are the orthogonal Laurent Jacobi polynomials considered in \cite{dimitrov ranga 2002}. It is important to remark some recent contributions especially \cite{Kar Gochayat q-Legueree zeros 2022,Kar JOrdaan Gochayat 2020} regarding distribution of zeros of $q$- Leguerre polynomials and some new hypergeometric type functions. With $\mathcal{P}^{(a,c)}_{-1}(x)=0$ and $\mathcal{P}^{(a,c)}_{0}(x)=1$, we get the monic polynomials
\begin{align*}
	\mathcal{P}^{(a,c)}_n(x) = \dfrac{(c)_n}{(1-a)_n}F(-n,1-a,1-c-n;x).
\end{align*}  

%The zeros of $\mathcal{P}_{6}(x)$ are $1.049267646, 1.179086731, 1.432298501, 1.918525281, 2.955569922, 5.865251919$ and the zeros of $\mathcal{P}_{6}(x;-2)$ are $.1082567303, 1.096706355, 1.381373784, 1.828789066, 2.307145896, 5.677728168$.  

%\begin{table}[h]
\begin{wraptable}{r}{8cm}
\caption{}\label{T1}
\renewcommand{\arraystretch}{1.03}
\centering
\begin{tabular}{|p{2.8cm}|p{3.6cm}|}
	\hline
	Zeros of $\mathcal{P}_{6}(x)$& Zeros of $\mathcal{P}_{6}(x;-2)$\\
	\hline
\hfill	1.049267646 &\hfill 0.108256730\\
	\hline
\hfill	1.179086731 &\hfill 1.096706355\\
	\hline
\hfill	1.432298501 &\hfill 1.381373784\\
	\hline
\hfill	1.918525281 &\hfill 1.828789066\\
	\hline
\hfill	2.955569922 &\hfill 2.307145896\\
	\hline
\hfill	5.865251919 &\hfill 5.677728168\\
	\hline
\end{tabular}
\end{wraptable}
%\end{table}
It was shown \cite{dimitrov ranga 2002} that $c_n$ and $\lambda_n$ are positive when $c>a>N$ or $a<c<1-N$, $n=1,2, \ldots N-1$. This example is used to investigate the behaviour of zeros for the following two cases: \par 
\textbf{Case I:}  $c>a>N$ and $\mu_k < 0$. \par
Perturbing the recurrence coefficient $c_n$ such that $c_4 \rightarrow c_4+\mu_4$, where $\mu_4 = -2$ generates a new sequence of polynomials $\mathcal{P}_{n}(x;-2)$. Note that $\mathcal{P}_{6}(x)$ and $\mathcal{P}_{6}(x;-2)$  have no common zeros as shown in \Cref{T1}. Clearly, these zeros satisfy the interlacing property in accordance with \Cref{interlacing theorem mu}. All the calculations are performed and figures are drawn using Mathematica on a system with an Intel Core i3-6006U CPU @ 2.00 GHz and 8 GB of RAM. %Here, the red diamonds represent the zeros of $\mathcal{P}_{n}(x)$ and blue circles represent the zeros of $\mathcal{P}_{n}(x;\mu_k)$. 

\begin{figure}[H]
	\includegraphics[scale=0.8]{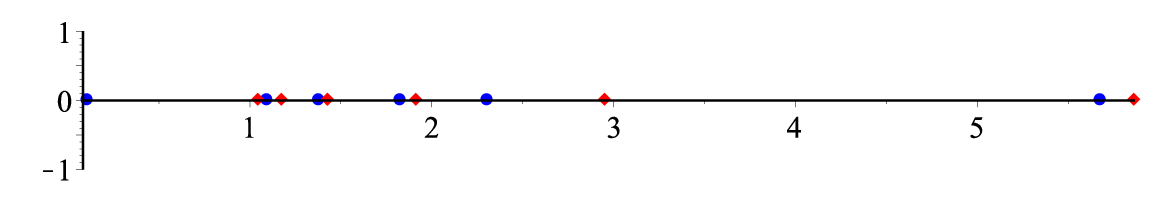}
	\caption{Zeros of $\mathcal{P}_{6}(x)$ (red diamonds) and $\mathcal{P}_{6}(x;-2)$ (blue circles) with parameter values $a=11$ and $c=12$.}
	\label{Fig1}
\end{figure}

%\begin{table}[h]
\textbf{Case II:}  $a<c<1-N$ and $\mu_k > 0$.\par
For this purpose, the recurrence coefficient $c_3$ is perturbed as $c_3 \rightarrow c_3+0.5$, resulting in a new sequence of polynomials $\mathcal{P}_{5}(x;0.5)$. The zeros of $\mathcal{P}_{5}(x)$ and $\mathcal{P}_{5}(x;0.5)$ are listed in \Cref{T2}. Now let us look (see \Cref{Fig2}) at the location of these zeros. The results follows \Cref{interlacing theorem mu}.

\begin{figure}[H]
	\includegraphics[scale=0.8]{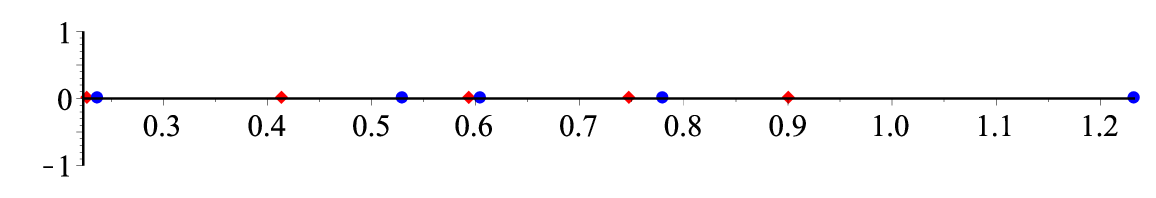}
	\caption{Zeros of $\mathcal{P}_{5}(x)$ (red diamonds) and $\mathcal{P}_{5}(x;0.5)$ (blue circles) with parameter values $a=-12$ and $c=-10$.}
	\label{Fig2}
\end{figure}
%\begin{remark}
%From the observation made above, it can be concluded that the zeros of L-Jacobi polynomials and co-recursive L-Jacobi polynomials polynomials interlace.
%\end{remark} 

\begin{wraptable}{r}{8cm}
\caption{}\label{T2}
\renewcommand{\arraystretch}{1}
\centering
\begin{tabular}{|p{2.5cm}|p{3.3cm}|}
	\hline
	Zeros of $\mathcal{P}_{5}(x)$& Zeros of $\mathcal{P}_{5}(x;0.5)$\\
	\hline
\hfill	0.2275033589 &\hfill 0.2372878714\\
	\hline
\hfill	0.4145169935 &\hfill 0.5301738824\\
	\hline
\hfill	0.5944664499 &\hfill 0.6051319225\\
	\hline
\hfill	0.7482316477 &\hfill 0.7803963908\\
	\hline
\hfill	0.9015129657 &\hfill 1.2332413480\\
	\hline
\end{tabular}
\end{wraptable}

To illustrate \Cref{Corolary monotonicity R1 pp}, the recurrence relation satisfied by L-Jacobi polynomials \eqref{L-jacobi} is taken into consideration.

\textbf{Case I:} $\mu_k > 0$ and $\mu_{k+1}>0$. \par
The zeros of perturbed L-Jacobi polynomials $\mathcal{P}_{5}(x;\mu_3=0.3,\mu_4=0.4)$ obtained by perturbing two successive terms $c_3$ and $c_4$ are listed in \Cref{T3 R1}. These zeros are represented by blue circles in \Cref{Fig3}. Now, take $\mu'_k > \mu_k$ and $\mu'_{k+1} > \mu_{k+1}$. The zeros of $\mathcal{P}_{5}(x;\mu'_3=0.5,\mu'_4=0.6)$ are plotted with green diamonds. The following point plot states the validity of our result for this case.
%$c_3 \rightarrow c_3+0.3$ and $c_4 \rightarrow c_4+0.4$
\begin{figure}[H]
	\includegraphics[scale=0.8]{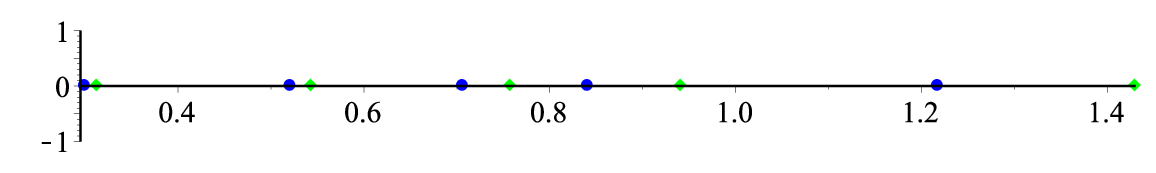}
	\caption{Zeros of $\mathcal{P}_{5}(x;0.3,0.4)$ (blue circles) and $\mathcal{P}_{5}(x;0.5,0.6)$ (green diamonds) with parameter values $a=-12$ and $c=-10$.}
	\label{Fig3}
\end{figure}

\begin{table}[h]
\caption{}
\renewcommand{\arraystretch}{1.05}
\label{T3 R1}
\centering
\begin{tabular}{|p{3cm}|p{4.8cm}|p{4.8cm}|}
	\hline
\hfill	Zeros of $\mathcal{P}_{5}(x)$&\hfill Zeros of $\mathcal{P}_{5}(x;0.3,0.4)$& \hfill Zeros of $\mathcal{P}_{5}(x;0.5,0.6)$\\
	\hline
\hfill	0.2275033589 &\hfill 0.299992667  &\hfill 0.313211285 \\
	\hline
\hfill	0.4145169935 &\hfill 0.521130362  &\hfill 0.543684240 \\  
	\hline
\hfill	0.5944664499 &\hfill 0.706604279   &\hfill 0.757953896 \\
	\hline
\hfill	0.7482316477 &\hfill 0.840953833 &\hfill 0.941308738 \\
	\hline
\hfill	0.9015129657 &\hfill 1.217550273 &\hfill 1.430073255 \\
	\hline
\end{tabular}
\end{table}
\textbf{Case II:} $\mu_k < 0$ and $\mu_{k+1}<0$. \par
The zeros of initial L-Jacobi polynomials $\mathcal{P}_{6}(x)$ and perturbed ones, i.e., $\mathcal{P}_{6}(x;\mu_4=-0.2,\mu_5=-0.25)$ (blue circles) and $\mathcal{P}_{6}(x;\mu'_4=-0.7,\mu'_5=-0.8)$ (green diamonds) are plotted in \Cref{Fig4} and given in \Cref{T4 R1}. Observe that the result holds in this case as well.

\begin{figure}[!htb]
	\includegraphics[scale=0.8]{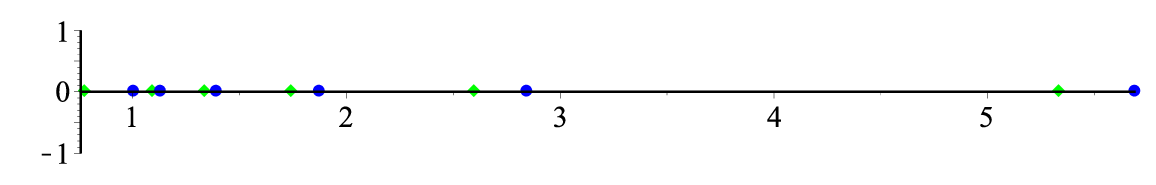}
	\caption{Zeros of $\mathcal{P}_{6}(x;-0.2,-0.25)$ (blue circles) and $\mathcal{P}_{6}(x;-0.7,-0.8)$ (green diamonds) with parameter values $a=11$ and $c=12$.}
	\label{Fig4}
\end{figure}

\begin{table}[h]
\caption{}
\renewcommand{\arraystretch}{1.1}
\label{T4 R1}
\centering
\begin{tabular}{|p{3cm}|p{4.8cm}|p{4.8cm}|}
	\hline
\hfill	Zeros of $\mathcal{P}_{6}(x)$& Zeros of $\mathcal{P}_{6}(x;-0.2,-0.25)$&Zeros of $\mathcal{P}_{6}(x;-0.7,-0.8)$\\
	\hline
\hfill	1.049267646 &\hfill 1.007621277  &\hfill 0.779670549 \\
	\hline
\hfill	1.179086731 &\hfill 1.133428374  &\hfill 1.096706319 \\  
	\hline
\hfill	1.432298501 &\hfill 1.394713267  &\hfill 1.340698443 \\
	\hline
\hfill	1.918525281 &\hfill 1.876397431 &\hfill 1.745550775 \\
	\hline
\hfill	2.955569922 &\hfill 2.846890957 &\hfill 2.601438104 \\
	\hline
\hfill	5.865251919 &\hfill 5.690948695 &\hfill 5.335935809 \\
	\hline
\end{tabular}
\end{table}

%The zeros of $\mathcal{P}_{8}(x)$ comes out to be $.4069296692, 1.843070331, 0.4116240276, .1710311797,\break .7690388590, 1.261114831, 2.415792909, 2.841859819$ and are plotted with red diamonds. When a perturbation $c_3 \rightarrow c_3-0.5$ and $\lambda_{3} \rightarrow 2\times \lambda_{3} $ is made, the zeros of $\mathcal{P}_{8}(x;-0.5,2)$ are $.1516686302, .3578291909, .6942294338, 1.297084073, 1.771409201, 2.454426431, 2.936149025,\break -.1627959860$ represented by blue circles in \Cref{Fig5}. 
\subsection{Note} It has been shown that zeros of co-recursive $R_I$ polynomials and unperturbed polynomials exhibit nice interlacing and monotonicity

\begin{wraptable}{r}{9cm}
	\caption{}\label{T5}
\renewcommand{\arraystretch}{1.1}
\centering
\begin{tabular}{|p{3cm}|p{4.5cm}|}
	\hline
\hfill	Zeros of $\mathcal{P}_{6}(x)$& Zeros of $\mathcal{P}_{6}(x;-0.5,2.5)$\\
	\hline
\hfill	1.049267646 &\hfill 0.828201058 \\
	\hline
\hfill	1.179086731 &\hfill 1.122836473\\  
	\hline
\hfill	1.432298501 &\hfill 1.445822282\\
	\hline
\hfill	1.918525281 &\hfill 1.829440364\\
	\hline
\hfill	2.955569922 &\hfill 3.134029425\\
	\hline
\hfill	5.865251919 &\hfill 5.861098970\\
	\hline
\end{tabular}

%\end{table}
\end{wraptable}

%\begin{example}
%For the purpose of illustration, we would recall the recurrence relation \eqref{Example 1 R1 recurrence} discussed in \Cref{example 1 R1 paper}.
%\end{example}

\noindent properties. However, this may not hold in the case of co-modified polynomials. The zeros of $\mathcal{P}_{6}(x)$ and $\mathcal{P}_{6}(x;-0.5,2.5)$ are listed in \Cref{T5} when a perturbation $c_3 \rightarrow c_3-0.5$ and $\lambda_{3} \rightarrow 2.5\times \lambda_{3} $ is made. Clearly, the zeros do not show any kind of interlacing for this case (see \Cref{Fig5}). Similarly, one can check that co-dilated $R_I$ polynomials do not have such a relationship with the unperturbed ones.
\begin{figure}[!htb]
	\includegraphics[scale=0.8]{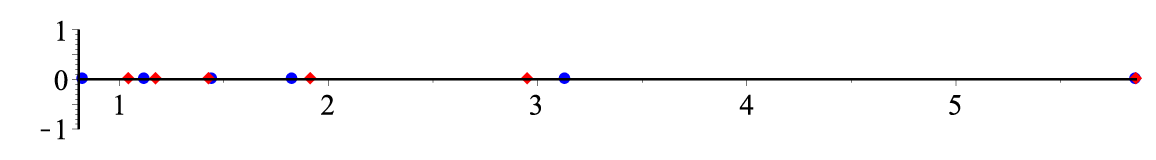}
	\caption{Zeros of $\mathcal{P}_{6}(x)$ (red diamonds) and $\mathcal{P}_{6}(x;-0.5,2.5)$ (blue circles) with parameter values $a=11$ and $c=12$.}
	\label{Fig5}
\end{figure}
\begin{remark}
This observation raises the question about the assumptions on $\mu_k$ and $\nu_k$ to be made such that the co-modified (or co-dilated) $R_I$ polynomials do show some interlacing with the unperturbed polynomials. Although, this subject is not dealt with in this manuscript and attracts further research. An idea to proceed in this direction can be found out in \cite{Liu Ren 2019, Liu Qi 2020}
\end{remark}
%\begin{remark}
%The validity of \Cref{interlacing theorem mu} when $\mathcal{P}_{n}(x)$ and $\mathcal{P}_{n}(x;\mu_k,\nu_k)$ have common zeros can also be established in similar manner using \Cref.
%\end{remark}

%A modification $\mu = \dfrac{c+n}{b+n}$ at the $n-$th level yields
%\begin{align*}
%	\mathcal{P}_{n+1}(z;\mu) &= \mathcal{P}_{n+1}(z)-\mu\mathcal{P}_{n}(z) \\
%	&= \dfrac{(c+1)_n}{(b+1)_n}(1-z)F(-n,b+1;c+1;1-z).
%\end{align*}
%It was shown \cite{KKB Swami 2018} that  $\mathcal{P}_{n+1}(z;\mu)$ satisfies mixed recurrence relation and are related to para-orthogonal polynomials on unit circle. Moreover, biorthogonality of such polynomials has also been discussed in the same. For more details, we refer to \cite{KKB Swami 2018} and references therein.
\subsection{Open problem} In \Cref{Structural relations} and \Cref{Inclusion interlacing and monotonicity}, the behaviour of zeros has been discussed for a particular case when $a_n=0$, $c_n \geq 0,$ and $\lambda_{n} > 0$ in recurrence relation \eqref{R1}, which corresponds to the L-orthogonal polynomials. It would be interesting to investigate the behaviour of the zeros for the case $a_n \neq 0$. More precisely, we want to know what conditions to impose on the non-zero parameters $a_n$, $\lambda_{n}$, and $c_n$ so that the zeros of the corresponding orthogonal polynomials lie on the real line and then investigate the interlacing relation between the perturbed and the original $R_I$ polynomials. At present, no concrete example is available in the literature for the case $a_n \neq 0$, and results for obtaining these conditions are still open, which, if found, will provide interesting consequences. A potential solution to address these situations can be discovered in \cite{Castillo Constr. Approx. 2022}.
 
\subsection*{Acknowledgments} The authors sincerely thank the anonymous referees for their constructive criticism, which helped to improve the manuscript. The first author would like to express his gratitude for the resources and support provided by Bennett University, Greater Noida, India, during the revision of this manuscript. This research work of the second author is supported by the Project No. CRG/2019/000200/MS of Science and Engineering Research Board, Department of Science and Technology, New Delhi, India.

\end{document}